\documentclass{amsart}
\usepackage{amsmath,amsthm,amsfonts}
\usepackage{mathtools}
\usepackage{mathrsfs}
\usepackage{nicematrix}
\usepackage{csquotes}
\usepackage{fontawesome}
\usepackage[dvipsnames]{xcolor}
\usepackage{tikz}
\usepackage{hyperref}

\usepackage{comment}

\usepackage[english]{babel}

\usepackage[backend=biber,sorting=nty,sortcites=true,style=alphabetic,
doi=true,natbib=true,isbn=false,url=false,giveninits=true,maxbibnames=99]{biblatex}
\renewbibmacro{in:}{}

\theoremstyle{plain}

\newtheorem{thm}{Theorem}
\newtheorem{cor}[thm]{Corollary}

\newtheorem{prop}[thm]{{\bf Proposition}}
\newtheorem{lem}[thm]{{\bf Lemma}}
\numberwithin{thm}{section}

\newcounter{hyp_counter}
\theoremstyle{definition}

\newtheorem{definition}[thm]{Definition}

\theoremstyle{remark}
\newtheorem{rem}{Remark}

\newcommand{\Aut}{\operatorname{Aut}}
\newcommand{\circled}[1]{\raisebox{.5pt}{\textcircled{\raisebox{-.9pt} {#1}}}}
\newcommand{\diag}{\operatorname{diag}}
\newcommand{\E}{\mathbb{E}}
\newcommand{\EP}{h}
\newcommand{\EPIA}{h_\text{FA}}
\newcommand{\grass}{\mathcal G}
\newcommand{\Id}{\operatorname{Id}}
\newcommand{\Diff}{\operatorname{Diff}}
\newcommand{\dKL}{d_\text{KL}}
\newcommand{\GL}{\operatorname{GL}}
\newcommand{\Gr}{\operatorname{Gr}}
\newcommand{\Homeo}{\operatorname{Homeo}}
\newcommand{\Jac}{\operatorname{Jac}}
\newcommand{\lin}{\operatorname{lin}}
\newcommand{\MF}{\operatorname{MF}}
\newcommand{\N}{\mathbb{N}}
\newcommand{\PF}{\operatorname{PF}}
\newcommand{\proj}{\operatorname{proj}}
\newcommand{\PW}{\text{PW}}
\newcommand{\R}{\mathbb{R}}
\newcommand{\RP}{\mathbb{R}\operatorname{P}}
\newcommand{\SL}{\operatorname{SL}}
\newcommand{\SLpm}{\operatorname{SL}^\pm}
\newcommand{\SO}{\operatorname{SO}}
\newcommand{\supp}{\operatorname{supp}}
\newcommand{\sval}{s}
\newcommand{\Svals}{S}

\newcommand{\V}{\mathbf V}
\newcommand{\vol}{\operatorname{vol}}
\newcommand{\var}{\operatorname{Var}}

\newcommand{\wt}[1]{\widetilde{#1}}

\newcommand{\abs}[1]{\left| #1\right|}
\newcommand{\mc}[1]{\mathcal{#1}}

\makeatletter
\def\blfootnote{\xdef\@thefnmark{}\@footnotetext}
\makeatother

\title[Effective Gaps Between Singular Values]{Effective Gaps between singular 
values of non-stationary matrix 
products subject to non-degenerate noise}
\author{Sam Bednarski}
\address{Department of Mathematics and Statistics, University of Victoria, Victoria, Canada}
\email{bednarskisam(a)gmail.com}
\author{Jonathan DeWitt}
\address{Department of Mathematics, The Pennsylvania State University, State College, PA, USA}
\email{dewitt@psu.edu}
\author{Anthony Quas}
\address{Department of Mathematics and Statistics, University of Victoria, Victoria, Canada}
\email{aquas(a)uvic.ca}
\date{\today}

\begin{document}

\begin{abstract}
We study the singular values and Lyapunov exponents of non-stationary 
random matrix products subject to small, absolutely continuous, additive 
noise. Consider a fixed sequence of matrices of bounded norm. Independently 
perturb the matrices by additive noise distributed according to Lebesgue 
measure on matrices with norm less than $\epsilon$. Then the gaps between 
the logarithms of the singular values of the random product of $n$ of these 
matrices are all of order at least $\epsilon^2n$, both in expectation; and almost surely 
for large $n$.

To prove this, we develop recent work of Gorodetski and Kleptsyn \cite{gorodetski2023nonstationary}.
That paper gives a very flexible method, based on relative entropy,
for showing that a non-stationary product of matrices in $\SL(d,\R)$ has a 
strictly positive Lyapunov exponent. We extend their work in two ways, firstly
by making the estimates quantitative in the context of absolutely continuous distributions,
giving the universal estimates described above; and secondly by developing a fibered version 
of their methods, working on flag bundles instead of the projective space to estimate 
gaps between arbitrary consecutive exponents. Our methods retain much of the flexibility
of those of Gorodetski and Kleptsyn, and we hope that they will find application in other related
problems.
\end{abstract}

\maketitle

\tableofcontents

\section{Introduction}
In this paper, we study non-stationary products of matrices subject to small, 
robust noise. Our main result says, informally, that if the noise is of size 
$\epsilon$, and we multiply $n$ matrices together, then the logarithm of the gaps 
between the singular values of 
the matrices will be of order at least $n\epsilon^2$. 
Our main result is the following. Write $s_k(A)$ for the $k$th largest singular value of the matrix $A$.

\begin{thm}\label{thm:main_thm_simple}
Let $\mu$ be an absolutely continuous probability measure on 
$M_{d\times d}(\R)$, the space of $d$ by $d$ matrices, with continuous 
density $\phi$ such that for some $C>0$, 
$\phi(A)\le C/\|A\|^{d^2+1}$ for all $A\in M_{d\times d}(\R)$. 
There exists $C_{\phi}$ with the following property.
For any $M\ge 1$, any $\epsilon\le 1$ and any sequence of matrices $(A_n)$ 
with $\|A_n\|\le M$, let $(E_n)$ be a sequence of independent identically distributed 
matrix random variables, with distribution $\mu$. If $B^n$ is the matrix
random variable $(A_n+\epsilon E_n)\cdots (A_1+\epsilon E_1)$, then
almost surely
\begin{equation}
\liminf_{n\to \infty} \tfrac 1n\big(\log \sval_k(B^n)-
\log \sval_{k+1}(B^n)\big)\ge \epsilon^2 C_\phi M^{-2}.
\end{equation}
\end{thm}
Results of \cite{bednarski2025lyapunov} show that this bound is optimal
up to a multiplicative constant by considering a sequence
of perturbations of the identity. See also the appendix for related results. 
If the matrix is not the identity then it may be possible to obtain gaps
between exponents that are larger than $\epsilon^2$ although we do not study this here. 
Theorem \ref{thm:main_thm_simple} also improves on the earlier result in \cite{atnip2023universal}, which 
gives an explicit, but weaker, bound for $\liminf \frac 1n\log (s_1(B^n)/s_d(B^n))$; and
also showed the existence of non-explicit strictly positive lower bounds for
$\liminf \frac 1n\log(s_k(B^n)/s_{k+1}(B^n))$.
Later in the 
paper we will state some generalizations of Theorem \ref{thm:main_thm_simple};
in particular, it is not 
important for our technique that the same noise be applied at each step. 
In Theorem \ref{thm:main_thm_abstract} we give a more abstract statement 
from which the theorem above follows, which also gives that 
\begin{equation}
\mathbb{E}\left[\log \sval_k(B^n)-
\log \sval_{k+1}(B^n)\right]\ge n\epsilon^2C_{\phi}M^{-2}.
\end{equation}

One major motivation for this work is to develop a technique for studying ergodic 
and statistical properties of perturbations of diffeomorphisms under robust noise 
\cite{kifer1988random}. The work of Young emphasizes the important role of random 
perturbations of random dynamical systems. This perspective is quite natural from the 
point of view of smooth ergodic theory as zero noise limits give rise to 
SRB measures \cite{cowieson2005SRB}, a phenomenon known as stochastic stability 
\cite{young1986stochastic}. This has been studied in many settings, for example, 
random perturbations of Henon maps \cite{benedicks2006random}.
See also \cite{young1986random} \cite{young2008chaotic}. For other works related 
to random perturbations of diffeomorphisms see 
\cite{blumenthal2017lyapunov,blumenthal2018lyapunov,blumenthal2022positive} 
producing non-zero Lyapunov exponents for perturbations of the standard map. 

One important thing we should observe is that Theorem \ref{thm:main_thm_simple} 
has limited content when the matrices $A_n$ already have some amount of hyperbolicity 
to begin with. In such a case, the Lyapunov exponents will survive under 
perturbations. See for example, Ruelle \cite{ruelle1979ergodic}, Peres \cite{peres1991analytic}. 

\subsection{Related Results}
Quantitative estimates 
on Lyapunov exponents for small perturbations of stochastic differential
equations have a long history. Some SDE of interest are perturbations of 
non-stochastic differential equations by noise, and hence it is quite natural to study the 
Lyapunov exponents of the perturbation.
In \cite{auslender1982asymptotic,pinksy1988lyapunov}, an exact formula is given 
for the top Lyapunov exponent for random perturbations of a nilpotent system 
and it is obtained that the exponent grows like $\epsilon^{2/3}$ for noise of size $\epsilon$.
For comparison, the paper \cite{mohammed1997lyapunov} obtains a related quantitative 
$\epsilon^2$ bound. For other types of systems, other types of expansions have also been  
obtained, see for example  
\cite{baxendale2001lyapunov,baxendale2022lyapunov,deville2011stability} 
where perturbations of Hamiltonian systems are also considered.
See also \cite{sowers2001tangent,baxendale2001lyapunov}.
Estimate on the size of Lyapunov exponents for a SDE play a role in a number 
of recent works such as 
\cite{bedrossian2024chaos,bedrossian2022regularity,bedrossian2024quantitative,chemnitz2023positive}. 
Many of these works obtain a  detailed study of the Lyapunov exponents by means of 
Furstenberg's formula. 
As mentioned above the papers \cite{bednarski2025lyapunov} and \cite{atnip2023universal} 
also give explicit bounds in the context of random matrix products.
 
The most closely related non-quantitative works to this paper are  \cite{goldsheid2022exponential} and 
\cite{gorodetski2023nonstationary}. 
In those works independent proofs of a non-stationary version of Furstenberg's 
theorem are obtained. Our approach builds on the framework of the second paper. 
In \cite{gorodetski2023nonstationary}, the authors introduce a framework 
for showing the effective growth of Lyapunov exponents. However, as they are 
working in much greater generality, they do not obtain the type of quantitative  
estimates on the growth rates that appear in Theorem \ref{thm:main_thm_simple}. 
However, their framework is so flexible that by developing a fibered version we are 
able to obtain such results. 

The study of Lyapunov exponents has a long history. We will focus on the setting 
of IID matrix products, to which this work bears the greatest similarity. 
The main theory follows the path laid out by Furstenberg \cite{furstenberg1963noncommuting}, who proved 
the most fundamental results on nontriviality of the Lyapunov exponents. 
Later to study simple spectrum, Guivarc'h and Raugi \cite{guivarch1984random} gave appropriate 
conditions for $\lambda_1>\lambda_2$. Later Gol'dsheid and Margulis gave a 
different argument for simplicity of Lyapunov spectrum that identified Zariski 
density of the measure $\mu$ as a key property \cite{goldsheid1989lyapunov}. 
The book \cite{benoist2016random} gives an overview of this theory and gives 
many additional references.
Obtaining effective, finite time estimates on Lyapunov exponents is not well 
understood, and the actual computation of Lyapunov exponents is non-trivial. 
See, for example the paper \cite{pollicott2010maximal}, which describes an 
algorithm for computing the Lyapunov exponents of an IID matrix product.

As far as the authors are aware, while there are many approaches to non-triviality
of Lyapunov exponents, there are not as many approaches to showing the simplicity 
of the Lyapunov spectrum. Most of the approaches to showing simplicity of Lyapunov 
spectrum essentially involve the construction of candidates for the Osceledec 
subspaces followed by a verification of their properties. See for example, 
\cite{viana2014lectures}, \cite{avila2007simplicity}, \cite{goldsheid1989lyapunov}. 
In order to construct these subspaces one first studies stationary measures on 
appropriate flag bundles as these stationary measures are essentially the 
distribution of the stable and unstable flags.

An important technique for showing non-triviality of Lyapunov exponents, but 
not necessarily simplicity, is the Invariance Principle in its various forms. 
See \cite{viana2014lectures} for an overview and \cite{furstenberg1963noncommuting}, 
\cite{ledrappier1986positivity}, \cite{baxendale1989lyapunov}, \cite{viana2008almost}, 
\cite{avila2010extremal}, \cite{viana2008almost}, where  various versions of 
the Invariance Principle appear. Informally, the Invariance Principle says 
that if the Lyapunov exponents of a measure $\mu$ along the fibers of a skew 
product over a hyperbolic base are all equal to zero, then the disintegration 
of $\mu$ along the fibers of the skew product are holonomy invariant. 
Consequently, if there is no holonomy invariant measure, then there must be 
non-trivial Lyapunov exponents along the fiber. In some sense, the approach of 
\cite{gorodetski2023nonstationary} fits with these works, but adds some extra 
flexibility by using that it  only takes one iterate of the dynamics to observe 
invariance. One can think of the approach of \cite{gorodetski2023nonstationary} 
as a type of effective Invariance Principle, because it shows that when there 
is not an invariant stationary measure for random dynamics, then every measure 
has uniformly positive entropy. See also \cite{gorodetski2024central}.

One of the difficulties with using the Invariance Principle to study 
simplicity of Lyapunov spectrum is that, directly, it only gives that the 
top and bottom Lyapunov exponents are different. Once the Lyapunov exponents 
are non-zero it stops giving new information. Of course, when studying 
$\SL(2,\R)$ cocycles, the invariance principle is enough for simplicity.

In some sense, the approach we develop here is an attempt to adapt the 
Invariance Principle approach in order to study simplicity of Lyapunov
exponents. In order to be able to apply this type of Invariance Principle 
argument, we set up an auxiliary problem where we study an induced action 
of the random dynamics on a bundle where every fiber is a circle. 
(This is the bundle of miniflags defined below). The linear dynamics acts 
on the circles. What we prove is that the amount that the average circle 
gets distorted will estimate quite precisely the gap between two singular values.
In this sense, some of the most technically similar works are those of 
Ledrappier and Lessa \cite{ledrappier2024exact},\cite{ledrappier2023exact}, 
\cite{lessa2021entropy}, who also make use of bundles with $1$-dimensional 
fibers that are well adapted to the problem they are studying and also make 
use of fibered Invariance Principles.  

\subsection{Outline of the Approach}
The approach in this paper owes a great debt to the approach of Gorodetski 
and Kleptsyn to the non-stationary Furstenberg theorem \cite{gorodetski2023nonstationary}, 
and in the case of absolutely continuous noise, we are able to extend their 
approach so that it can effectively treat all Lyapunov exponents. 

In order to explain our contribution, let us first explain why the argument in 
\cite{gorodetski2023nonstationary} does not immediately extend to getting this 
type of effective estimate. The main structure of the argument in that paper 
has two steps. First, they develop an enhanced theory of the Furstenberg entropy, 
which says entropy is additive under convolution and hence increases under 
the hypotheses of their main result. 
Then they show that positive entropy implies that volume on $\RP^{d-1}$ under the linear  
random walk is distorted at least as much as the 
entropy predicts. If a matrix $A$ distorts volume, this implies there is a gap 
between $\sigma_1(A)$ and $\sigma_d(A)$. Estimating how large this gap is thus 
requires estimating the entropy of certain random dynamics from below. 
However, having solved that problem it does not 
obviously give a way to control the gap between $\sigma_1(A)$ and $\sigma_2(A)$.

In order to deal with this issue, we find a different space to let a matrix 
act on other than projective space. The space we choose is the space of partial 
flags of the type $(k-1,k,k+1)$, which we call \emph{miniflags} (Subsec.~\ref{subsec:miniflags}). 
These miniflags form a bundle over the space partial flags of type $(k-1,k+1)$ 
and we consider the action induced on the fibers of this bundle. A key 
technical result is the fact that if $A$ acts on the bundle of miniflags and we consider the 
distortion of volume on the miniflags, and average this distortion over all miniflags, then 
this closely approximates $\sigma_{k}/\sigma_{k+1}$. This is proved in Proposition \ref{prop:gapest}. 

The next task is to show that a typical random word in fact distorts 
volume along miniflags. To do this, in Section \ref{sec:entropy_of_bundle_maps} 
we introduce a generalization of the 
Furstenberg entropy that is adapted to fibered systems. This averages 
the entropy over all the different images of a given fiber.
We then develop an analogous entropy theory in this setting showing, for example, 
additivity of the entropy under appropriate hypotheses. Ultimately, this shows 
that if this fibered entropy on $(k-1,k+1)$ miniflags is large, 
then so too will be the gap between $\sigma_k$ and $\sigma_{k+1}$ due to Proposition \ref{prop:gapest}.

The final thing we must do to conclude is obtain quantitative bounds on the fiberwise 
entropy, showing \emph{uniform fiberwise push-forward entropy}. For this we 
introduce something called the pointwise entropy (Def.~\ref{defn:pointwise_entropy}). 
The pointwise entropy measures how much the Jacobians of different maps between 
two points vary. We then show in Proposition \ref{prop:fubini_argument_for_bundles} 
that one can use the pointwise 
entropy to obtain a lower bound on the entropy for the fiber bundle. Then in Section
\ref{sec:pointwise_entropy_est} we compute a lower bound on the pointwise entropy. 
In Section \ref{sec:conclusion} we then conclude the proof.

\smallskip
\textbf{Acknowledgments}. The first author was supported by the 
National Science Foundation under Award No.~DMS-2202967. The second author was supported
by a Discovery Grant from NSERC. The authors are also 
grateful to Agnieszka Zelerowicz and Zhenghe Zhang for organizing the special 
session on dynamical systems where the authors met and first undertook the study of this problem. 
We thank Reuben Drogin for helpful comments, and for pointing out an oversight in the computation of the densities of the conditional measures in an earlier version.

\section{Preliminaries}
In this section we introduce all of the basic structures that we will use 
throughout the rest of the paper.

\subsection{Miniflags}\label{subsec:miniflags}
In this paper we will study the induced action of a linear map $A\in \GL(d,\R)$ 
on certain space of flags. 
We will write $\Gr(k,d)$ for the Grassmannian of $k$-planes in $\R^d$. 
For $1\le k\le d-1$, we will write $\PF(k-1,k+1)$ for the space of partial flags 
of the form $0\subseteq V_{k-1}\subset V_{k+1}\subseteq \R^d$ where $\dim V_{i}=i$. 
Associated to a partial flag $(V_{k-1},V_{k+1})$ is the \emph{miniflag}, $\mc F(V_{k-1},V_{k+1})$,
of all subspaces $V_k$ satisfying $V_{k-1}\subset V_k\subset V_{k+1}$. 
We will say that the \emph{core dimension} of
$\mc F(V_{k-1},V_{k+1})$ is $k-1$. 
Note that the miniflags are the fibers of the obvious map $\PF(k-1,k,k+1)\to \PF(k-1,k+1)$ 
where we forget about the $k$-dimensional subspace. So, thinking of 
$\PF(k-1,k,k+1)$ as a fiber bundle over $\PF(k-1,k+1)$, the fibers are precisely the miniflags, and the 
miniflags are indexed by $\PF(k-1,k+1)$.
Note also that there is a natural correspondence between 
the miniflag $\mc F(V_{k-1},V_{k+1})$ and the (one-dimensional)
projective space of $V_{k+1}\ominus V_{k-1}$,
where $V_{k+1}\ominus V_{k-1}$ denotes $V_{k+1}\cap V_{k-1}^\perp$. 

\subsection{Basic Facts Concerning Entropy}
Below we will define a notion of entropy and introduce its properties.
Given a measure $\mu$ on $\Homeo(M,N)$, its entropy quantifies the smallest extent to which 
the measures $f_*\nu$ vary for $f$ distributed according to $\mu$. 
If there is a measure $\nu$ on $M$ such that $f_*\nu$ is 
the same for $\mu$-a.e.\ $f$, the entropy of the action of $\mu$ will be zero.
The definitions below appear in \cite[Sec.~2]{gorodetski2023nonstationary}. 

\begin{definition}\label{defn:KL}
Suppose that $M$ is a standard Borel space and that $\nu$ and $\wt{\nu}$ 
are Borel probability measures on $M$. We define the 
\emph{Kullback-Leibler divergence} of $\nu$ with respect $\wt{\nu}$ by
\[
\dKL(\nu\vert\wt{\nu})=\begin{cases}
\int_M \log \frac{d\nu}{d\wt{\nu}}\left(\frac{d\nu}{d\wt{\nu}}\right)\,d\wt{\nu}=
\int_M\log(\frac{d\nu}{d\wt{\nu}})\,d\nu &\text{ if }\nu\ll\wt{\nu};\\
+\infty & \text{ else}.
\end{cases}
\]
This is also known as the \emph{relative entropy}.
\end{definition}

Suppose that $M$ and $N$ are closed manifolds and that $\mu$ is a measure 
on $\Homeo(M,N)$. Given measures $\nu$ on 
$M$ and $\nu'$ on $N$, we can define the \emph{mean relative entropy} 
$\Phi_{\mu}(\nu\vert \nu')$ of these measures by
\begin{equation}
\Phi_{\mu}(\nu\vert\nu')=\int \dKL(f_*\nu\vert \nu')\,d\mu(f). 
\end{equation}

If $\mu$ is a probability measure on $\Homeo(M,N)$ and $\nu$ is a probability measure on $M$,
we define the convolution of $\mu$ with $\nu$ by:
\begin{equation}
\mu*\nu=\int_{\Homeo(M,N)} f_*\nu\,d\mu(f).
\end{equation}
As is often done, we also use the notation $\mu_2*\mu_1$ to refer to the 
push-forward of $\mu_2\otimes\mu_1$
under the map $(f_2,f_1)\mapsto f_2\circ f_1$.

A claim we shall use below is the following.
\begin{lem}\label{lem:min_Phi_nu_nu_prime}
\cite[Lem.~2.8]{gorodetski2023nonstationary} 
Suppose $\mu$ is a compactly supported probability on $\Homeo(M,N)$. 
For a fixed probability measure $\nu$ on $M$,
\[
\inf_{\nu'}\,\Phi_{\mu}(\nu\vert\nu')=\Phi_{\mu}(\nu\vert\mu*\nu),
\]
where the infimum is taken over probability measures on $N$.
That is, the mean relative entropy is minimized by the ``mean push-forward", $\nu'=\mu*\nu$ 
\end{lem}

\noindent This leads us to define the Furstenberg entropy as follows. 

\begin{definition}\label{defn:furstenberg_entropy}\cite[Def.~2.3]{gorodetski2023nonstationary}
Suppose that $M$ and $N$ are two closed manifolds and that $\mu$ is a probability 
measure on $\Homeo(M,N)$. We define the \emph{Furstenberg entropy} of $\nu$ by:
\[
\Phi_{\mu}(\nu)= \Phi_\mu(\nu|\mu*\nu)=\int \dKL(f_*\nu\vert \mu*\nu)\,d\mu(f).
\]
\end{definition}
This is a measure of the extent to which the push-forwards $f_*\nu$ 
vary as $f$ runs over the support of $\mu$.
In this notation, Lemma \ref{lem:min_Phi_nu_nu_prime} states $\Phi_\mu(\nu|\nu')\ge
\Phi_\mu(\nu)$ with equality if and
only if $\nu'=\mu*\nu$. 

\begin{definition}\label{defn:pushforward_entropy}
Suppose that $\mu$ is a probability measure on the space of homeomorphisms $\Homeo(M,N)$. 
Then we define the \emph{push-forward entropy} of $\mu$ (when
maps distributed as $\mu$ act on $M$) by
\[
\EP(\mu)=\inf_\nu\, \Phi_\mu(\nu),
\]
where the infimum is taken over probability measures on $M$.
\end{definition}
The push-forward entropy only vanishes if there is a probability measure $\nu$ on $M$ such that
$f_*\nu=\mu*\nu$ for $\mu$-a.e.\ $f$\cite[Lem.~2.5]{gorodetski2023nonstationary}. 

\begin{lem}[Concavity of entropy]\label{lem:convexity_entropy}
Suppose that $M$ and $N$ are closed manifolds and that $\mu$ is a 
compactly supported measure on $\Homeo(M,N)$. 
Suppose we represent $\mu$ as a convex combination of a family of probability 
measures $\{\mu_{\lambda}\}_{\lambda\in\Lambda}$ 
indexed by a probability space $(\Lambda,\eta)$ where each $\mu_{\lambda}$ is a 
probability measure on $\Homeo(M,N)$. 
In other words, $\mu=\int_{\Lambda} \mu_{\lambda}\,d\eta(\lambda)$. Then
\begin{equation}
\EP(\mu)\ge \int_{\Lambda} \EP(\mu_{\lambda})\,d\nu(\lambda). 
\end{equation}
\end{lem}
\begin{proof}
This follows from the inequalities stated above. In particular, we will repeatedly 
use that $\Phi_{\mu}(\nu\vert \nu')$ is minimized when $\nu'=\mu*\nu$ 
(Lemma \ref{lem:min_Phi_nu_nu_prime}). 
\begin{align*}
\EP(\mu)=\inf_\nu\,\Phi_{\mu}(\nu)&=\inf_{\nu}\,\Phi_{\mu}(\nu\vert \mu*\nu)\\
&=\inf_{\nu}\, \int_{\Homeo(M,N)} \dKL(f_*\nu\vert \mu*\nu)\,d\mu(f)\\
&=\inf_{\nu}\, \int_{\Lambda} \int_{f} \dKL(f_*\nu\vert \mu*\nu)\,d\mu_{\lambda}(f)\,d\eta(\lambda)\\
&=\inf_{\nu}\,\int_{\Lambda} \Phi_{\mu_{\lambda}}(\nu\vert \mu*\nu)\,d\eta(\lambda)\\
&\ge \int_{\Lambda} \inf_{\nu}\, \Phi_{\mu_{\lambda}}(\nu\vert \mu*\nu)\,d\eta(\lambda)\\
&\ge \int_{\Lambda} \inf_{\nu}\,\Phi_{\mu_{\lambda}}(\nu\vert \mu_{\lambda}* \nu)\,d\eta(\lambda)\\
&= \int_{\Lambda} \inf_{\nu}\, \Phi_{\mu_{\lambda}}(\nu)\,d\eta(\lambda)\\
&=\int_{\Lambda} \EP(\mu_{\lambda})\,d\eta(\lambda).
\end{align*}
\end{proof}

The following lemma is one of the key tools of \cite{gorodetski2023nonstationary}: 
the push-forward entropy of a convolution is superadditive. 

\begin{lem}[Superadditivity of entropy]\label{lem:additivity_of_entropy_gap}
Suppose that $\mu_1$ is a probability measure on $\Homeo(M_1,M_2)$ and that 
$\mu_2$ is a probability measure on $\Homeo(M_2,M_3)$. Then 
\[
\EP(\mu_2*\mu_1)\ge \EP(\mu_1)+\EP(\mu_2).
\]
\end{lem}

\begin{proof}
Proposition 2.10 in \cite{gorodetski2023nonstationary} gives that 
\[
\Phi_{\mu_2*\mu_1}(\nu)=\Phi_{\mu_1}(\nu)+\Phi_{\mu_2}(\mu_1*\nu).
\]
In particular, we see that 
\[
\EP(\mu_2*\mu_1)=\inf_{\nu}\,\Phi_{\mu_2*\mu_1}(\nu)\ge \inf_{\nu}\, 
\Phi_{\mu_1}(\nu)+\inf_{\nu'}\,\Phi_{\mu_2}(\nu')=\EP(\mu_1)+\EP(\mu_2).
\]
\end{proof}

\subsection{The relationship between entropy and volume distortion}

\begin{lem}[Entropy causes volume distortion \cite{gorodetski2023nonstationary}]
\label{lem:entropy_volume_distortion} 
Suppose that $M_1$ and $M_2$ 
are two Riemannian manifolds with unit volume, and that $\mu$ 
is a measure on $\Diff^1(M_1,M_2)$. Then 
\[
\int \log \mc{N}(f)\,d\mu(f)\ge \Phi_{\mu}(\vol)\ge \EP(\mu),
\]
where 
\[
\mc{N}(f)=\max_{x\in M_1} \abs{\Jac(f)_x}^{-1}
\]
is the maximum of the volume distortion of $f$.

\end{lem}

\begin{proof}
An identical argument appears in the proof of \cite[Thm.~1.6]{gorodetski2023nonstationary}. 
Write $m_1$ for the volume on $M_1$ and $m_2$ for the volume on $M_2$. Let 
\[
\mc{N}(f)=\max_{x\in M_1} \abs{\Jac(f)_x}^{-1}=\max_{y\in M_2} \frac{d(f_*m_1)}{dm_2}(y). 
\] 
Hence from the definition of $d_{KL}$ (Def.~\ref{defn:KL}),
\[
\log \mc{N}(f)\ge d_{KL}(f_*m_1\vert m_2). 
\]
Thus taking expectations over $f$, which is distributed according to $\mu$:
\[
\E_{\mu}\left[\log \mc{N}(f)\right]\ge \E_{\mu}\left[d_{KL}(f_*m_1\vert m_2)\right]
\ge \E_{\mu}\left[d_{KL}(f_*m_1\vert \mu*m_1)\right]\ge \Phi_{\mu}(m_1)\ge \EP(\mu),
\]
where we have used that $\mu*m_1$ minimized the mean of the Kullback-Leibler 
divergences in the second inequality. 
\end{proof}
In fact, we will only use the above formula for maps of one-dimensional projective spaces. 

\subsection{Miniflag volume distortion and singular values}
Let $\mc{F}(V_{i-1},V_{i+1})$ be a miniflag. The entropy of a map controls the mean volume
distortion. If $M_1$ and $M_2$ are two Riemannian manifolds, and 
$f\colon M_1\to M_2$ is a diffeomorphism, we will 
write $(J(f)^{-1})_{\max}$ for $\max_{x\in M_1} \abs{\det D_xf}^{-1}$; as usual 
$\abs{\det D_xf} = \|(D_xf)_* \vol_x\|/\|\vol_{f(x)}\|$.  

We begin with the following elementary lemma. 

\begin{lem}\label{lem:volume_distortion_RP1}
If $A\in \GL(2,\R)$ and $A_*$ is the induced map $\RP^1\to \RP^1$, then
\[
(J(A_*)^{-1})_{\max}=\sval_1(A)/\sval_2(A).
\]
\end{lem}
\begin{proof}
We show that for any $B\in\GL(d,\R)$, $\max_{x\in\RP^1}J(B_*)=\sval_1(B)/\sval_2(B)$. Applying this to 
$B=A^{-1}$, using $B_*=A_*^{-1}$, we obtain 
\begin{align*}
(J(A_*)^{-1})_\text{max}&=((J(A_*))_\text{min})^{-1}=
(J(B_*))_{\text{max}}\\
&=\sval_1(B)/\sval_2(B)=\sval_2(A)^{-1}/\sval_1(A)^{-1}=\sval_1(A)/\sval_2(A).
\end{align*}
Let $B\in\GL(d,\R)$. By the singular value decomposition, it suffices to consider 
the case of a diagonal matrix. 
Now consider the unit circle $S^1$ in $\R^2$. We may regard $B_*$ as the composition of two maps $B$ 
and the projection $P$ of $\R^2\setminus \{0\}$ to $S^1$. Thus it suffices to compute the norm of 
the action on a tangent vector to $S^1$. The norm of $B$ is $\sval_1(B)$. The norm of $P$ restricted 
to vectors $v\in \R^2$ with $\|v\|\ge \delta$ is $\delta^{-1}$. Thus for the composition 
\[
B_*: S^1\to B(S^1)\to BA(S^1)=S^1,
\]
one has $J_{\max}\le \sval_1(B)\sval_2(B)^{-1}$. A straightforward computation using the 
chart $\theta\mapsto [\theta,1]$ shows that this value, $\sval_1(B)\sval_2(B)^{-1}$ is realized as 
the differential at the point $[0,1]$. Thus we have that $J_{\max}(B_*)=\sval_1(B)\sval_2^{-1}(B)$,
as required. 
\end{proof}

We will need to estimate the volume distortion on the miniflags. For this we have the following lemma.

\begin{prop}
Suppose that $V$ is an inner product space. Then 
\[
\log \frac {\sval_1(B|_{V_{k+1}\perp V_{k-1}})}{\sval_2(B|_{V_{k+1}\perp V_{k-1}})} 
=\log J_{\max} (B\vert \mc{F}(V_{k-1},V_{k+1})).
\]
\end{prop}
\begin{proof}
The map $B_{V_{k+1}\perp V_{k-1}}$ is a linear map from $V_{k+1}\ominus V_{k-1}$ 
to $AV_{k+1}\ominus AV_{k-1}$. 
The map $B|\mc F(V_{k-1},V_{k+1})$ sending $\mc{F}(V_{k-1},V_{k+1})$ to $\mc{F}(BV_{k-1},BV_{k+1})$ 
is the projectivization of this map. Thus, this statement is immediate from Lemma 
\ref{lem:volume_distortion_RP1}.
\end{proof}

\section{Entropy of Random Bundle Maps}\label{sec:entropy_of_bundle_maps}
Below we will consider random bundle maps; we will be particularly 
interested in the restriction of these random maps to different fibers.
Suppose that $E$ is a fiber bundle over a manifold $M$. As mentioned earlier, we are primarily interested
in the case $M=\PF(k-1,k+1)$ and $E=\PF(k-1,k,k+1)$ so that the fiber over $(V_{k-1},V_{k+1})$ is the 
miniflag $\mc F(V_{k-1},V_{k+1})$. 
The miniflag $\mc F(V_{k-1},V_{k+1})$ inherits a Haar measure, which we denote as $m_{V_{k-1},V_{k+1}}$. 
Let $\Aut(E)$ denote the space of bundle maps $F\colon E\to E$. 
Given a probability measure $\mu$ on $\Aut(E)$, 
we let $\bar\mu$ be the associated quotient dynamics on $M$.

Suppose that $\{\mu_n\}_{n\in \N}$ is a sequence of probability measures on $\Aut(E)$. 
We will write $\mu^{*n}$ for the convolution $\mu_n*\cdots\mu_2*\mu_1$, 
that is $\mu^{*n}$ is the distribution of the 
(non-stationary) evolution from time 0 to time $n$.

We recall here that given a measurable partition $\mc{P}$ of a Borel space 
$\Omega$ and a Borel probability measure $\nu$ that there is a disintegration 
of $\nu$ along the partition and for each element $P\in \mc{P}$ we obtain a probability 
measure $\nu_P$ on $P$ such that for each Borel $A\subseteq \Omega$,
\[
\nu(A)= \int \nu_P(A\cap P)\,d\bar{\nu}(P),
\]
where $\bar{\nu}$ is the measure on $\Omega/\mc{P}$. For more details about this 
construction see the discussion surrounding \cite[Thm.~5.11]{viana2016foundations}.

Next we define the fiberwise entropy. As above, suppose that $\mu$ is a 
probability measure on $\Aut(E)$. 
Suppose that $\mc{F}_1$ is a fiber in $E$. The fiber-averaged push-forward entropy 
estimates over all the possible 
images of $\mc{F}_1$, what is the mean entropy of the maps conditioned on being that image. 

In general, the maps $f$ in the support of $\mu$ map $\mc F_1$ into multiple fibers.
For fibers $\mc F_1$ and $\mc F_2$, write $\mu_{\mc{F}_1,\mc{F}_2}$ for the 
conditional distribution of $\mu$ 
conditioned on $\mu$ carrying $\mc{F}_1$ to $\mc{F}_2$. Let $\bar{\mu}_{\mc{F}_1}$ be the 
distribution of the images of $\mc{F}_1$ (recall that points in $M$ correspond to these fibers). 
Then we define the \emph{fiber-averaged push-forward entropy} of the fiber 
$\mc{F}_1$ under the dynamics $\mu$ to be 
\begin{equation}\label{eqn:fiberwise_pushforward_entropy}
\EPIA(\mc{F}_1,\mu)=\int \EP(\mu_{\mc{F}_1,\mc{F}_2})\,d\bar{\mu}_{\mc{F}_1}(\mc{F}_2).
\end{equation}
Note that the integrand is measurable because $\mu_{\mc{F}_1,\mc{F}_2}$ is 
Borel-measurable, $h(\cdot)$ is Borel-measurable as it is lower-semicontinuous 
\cite[Cor.~2.6]{gorodetski2023nonstationary}, and the composition of Borel-measurable
functions is Borel-measurable. 

\begin{definition}\label{defn:uniform_fiberwise_pushforward_entropy}
Suppose that $\mu$ is a measure on $\Aut(E)$. We say that $\mu$ has 
\emph{uniform fiberwise push-forward entropy} $h>0$ if 
\begin{equation}
\inf_{\mc{F}\in E} \EPIA(\mc{F},\mu)= h.
\end{equation}
\end{definition}

\begin{prop}\label{prop:additivity_of_pushforward_entropy}
Suppose that $\mu_1$ and $\mu_2$ are probability measures on $\Aut(\mc{E})$. 
Suppose that $\mu_2$ has uniform 
fiberwise push-forward entropy $h$. Then, 
\begin{equation}
\EPIA(\mc F_0,\mu_2*\mu_1)\ge h+\EPIA(\mc F_0,\mu_1).\label{eqn:additivity}
\end{equation}
\end{prop}

\begin{proof}
This follows along the lines that we have considered above. First we will estimate the fiberwise entropy
between a pair of fibers. As before, let $\mu^{*2}=\mu_2*\mu_1$. Recall the above notation that 
$\mu_{1,\mc{F}_0,\mc{F}_1}$ is the conditional distribution of $\mu_1$ on the maps from 
$\mc{F}_0$ to $\mc{F}_1$, and $\mu_{2,\mc{F}_1,\mc{F}_2}$ is defined similarly 
for $\mu_2$. Let $\bar{\mu}_{1,\mc{F}_0}^{\mc{F}_2}$ denote the distribution 
of the image $\mc{F}_1$ of $\mc{F}_0$ given that the image of 
$\mc{F}_0$ under $\mu^{*2}$ is equal to $\mc{F}_2$. We observe that
\begin{equation}\label{eqn:mu_2_representation}
\mu^{*2}_{\mc{F}_0,\mc{F}_2}=\int \mu_{2,\mc{F}_1,\mc{F}_2}*
\mu_{1,\mc{F}_0,\mc{F}_1}\,d\bar\mu_{1,\mc{F}_0}^{\mc{F}_2}(\mc{F}_1). 
\end{equation}
This equality is a combination of two things. First, if we analogously define 
$\mu^{*2}_{\mc{F}_0,\mc{F}_1,\mc{F}_2}$ to be the distribution of $f=f_2f_1$ given 
that $f_1(\mc{F}_0)=\mc{F}_1$ and $f_2(\mc{F}_1)=\mc{F}_2$, then 
$\mu^{*2}_{\mc{F}_0,\mc{F}_1,\mc{F}_2}=\mu_{\mc{F}_1,\mc{F}_2}*\mu_{\mc{F}_0,\mc{F}_1}$. 
Secondly, we can represent $\mu^{*2}_{\mc{F}_0,\mc{F}_2}$ by conditioning on the 
middle point in the trajectory: $\mu^{*2}_{\mc{F}_0,\mc{F}_2}=
\int \mu_{\mc{F}_0,\mc{F}_1,\mc{F}_2}(f)\,d\eta(\mc{F}_1)$, where $\eta$ is the 
distribution of the flag $\mc{F}_1$. In \eqref{eqn:mu_2_representation} we have combined 
these two things and renamed $\eta$ to reflect that it is a conditioned version of $\mu_1$.

Thus by Lemma \ref{lem:convexity_entropy} it follows that
\[
\EP(\mc{F}_0,\mu^{*2}_{\mc{F}_0,\mc{F}_2})\ge 
\int \EP(\mc{F}_0,\mu_{2,\mc{F}_1,\mc{F}_2}*\mu_{1,\mc{F}_0,\mc{F}_1})\,
d\bar\mu_{1,\mc{F}_0}^{\mc{F}_2}(\mc{F}_1). 
\]
By assumption, $\mu_2$ has uniform fiberwise entropy gap $h$. 
So by integrating the previous line over $\mc{F}_2$, 
and applying Lemma \ref{lem:additivity_of_entropy_gap} in the fourth line below, we see that
\begin{equation}\label{eq:uniformsumbound}
\begin{split}
\EPIA(\mc F_0,\mu^{*2})&=\int \EP(\mc F_0,\mu^{*2}_{\mc{F}_0,\mc{F}_2})\,
d \bar{\mu}^{*2}_{\mc{F}_0}(\mc{F}_2)\\
&\ge \int \EP(\mc F_0,\mu_{2,\mc{F}_1,\mc{F}_2}*\mu_{1,\mc{F}_0,\mc{F}_1})\,
d\bar\mu_{1,\mc{F}_0}^{\mc{F}_2}(\mc{F}_1)\,
d\bar{\mu}^{*2}_{\mc{F}_0}(\mc{F}_2) \\
&=\int \EP(\mc F_0,\mu_{2,\mc{F}_1,\mc{F}_2}*\mu_{1,\mc{F}_0,\mc{F}_1})\,
d\bar\mu_{2,\mc{F}_1}(\mc{F}_2)\,d\bar\mu_{1,\mc{F}_0}(\mc{F}_1)
\\
&\ge \int \big(h+\EP(\mc F_0,\mu_{1,\mc{F}_0,\mc{F}_1})\big)\,
d\bar\mu_{1,\mc{F}_0}(\mc{F}_1)=h+\EPIA(\mc{F}_0,\mu_1). 
\end{split}
\end{equation}
\end{proof}

\section{Tools for Bounding Entropy Below}
In this section we introduce three basic tools that allow us to estimate 
the entropy of the random bundle map on 
miniflags, and then use this estimate on the entropy to estimate the gap between singular values of the 
random matrix product.

\subsection{Pointwise Entropy}
In this subsection, we develop a tool for actually estimating the fiberwise entropy gap.

Our main tool for actually bounding entropy is a bound on 
the entropy for absolutely continuous measures.  
The pointwise entropy measures the failure of volume to be infinitesimally 
invariant under the dynamics of $\mu$ at the point $x$.

\begin{definition}\label{defn:pointwise_entropy}
Suppose that $M$ and $N$ are smooth manifolds, $x\in M$, $y\in N$, and that 
$\mu$ is a probability measure on $\Diff^1(M,N)$ such that $\mu$-a.s.~$f(x)=y$. 
Then we define the \emph{pointwise (volume) entropy} of $\mu$ at $y$ by 
\[
h_{\PW}(y,\mu)\coloneqq \int \log\left(\frac{\Jac_y(f^{-1})}{\int \Jac_y(f^{-1})\,d\mu(f)}\right)
\frac{\Jac_y(f^{-1})}{\int \Jac_y(f^{-1})\,d\mu(f)}\,d\mu(f).
\]
\end{definition}

Note that the quantity in the above definition might equal $0$.
We now record a simple estimate that will be useful for estimating the pointwise entropy later.
\begin{lem}\label{lem:convexity_xlogx}
Let $\psi(x)=x\log x$. Suppose that $X\ge 0$ is a random variable 
such that $\abs{X}\le B$ for some $B>0$
and $\E X=1$. Then
\[
\E\left[\psi(X)\right]\ge \var(X)/(2B).
\]
\end{lem}
\begin{proof}
Let $\nu$ denote the distribution of $X$ on $(0,\infty)$. 
Then using Taylor's formula with remainder gives:
\begin{align*}
\E\left[\psi(X)\right]&=
\int (\psi(x)-\psi(1))\,d\nu(x)\\
&=\int\left(\psi'(1)(x-1)+\frac{\psi''(g(x))}{2}(x-1)^2\right)\,d\nu(x)\\
&=\int \frac{\psi''(g(x))}{2}(x-1)^2\,d\nu(x)
\end{align*}
Now note that $\psi''(x)=1/x$. Thus 
\[
\E\left[\psi(X)\right]\ge \frac{\var(X)}{2B}.
\]
\end{proof}

\begin{lem}\label{lem:variance_lower_bound}
Let $X$ be an absolutely continuous random variable taking values in an interval 
$[c,d]$ where the density of the distribution of $X$ satisfies $\rho(x)/\rho(x')\le b$ for all $x,x'\in [c,d]$.
Then $\var X\ge (d-c)^2/(32b^2)$. 
\end{lem}
\begin{proof}
Let $J$ be an interval of length $(d-c)/(2b)$ centered at $\E X$. 
We have $\mathbb P(X\in J)\le (d-c)\rho_\text{max}/(2b)\le (d-c)\rho_\text{av}/2=\frac 12$, 
because the average value of a unit mass density is $1$. 
Since $|X-\E X|^2\ge [(d-c)/(4b)]^2$ on $J^c$, we see
$\var X=\E(X-\E X)^2\ge (d-c)^2/(32b^2)$ as claimed.
\end{proof}

\subsection{The Fubini Argument}
Suppose that $M$ and $N$ are two manifolds with volume and that $\mu$ is a 
measure supported on $\Diff(M,N)$. Suppose that $\mu$ is smoothing in the sense 
that for any $x\in M$, $\mu*\delta_x$ is absolutely continuous. How can we use the 
pointwise entropy to bound the entropy gap of $\mu$ below? By definition, the 
push-forward entropy $\Phi_{\mu}(\nu)$ is given by evaluating the integral:
\[
\int_{\Diff} h(f_*\nu\vert \mu*\nu)\,d\mu(f)=\int_{\Diff}\int_N 
\log\left(\frac{df_*\nu}{\,d\mu*\nu\,}(y)\right)\frac{df_*\nu}{\,d\mu*\nu\,}(y)\,
d\mu*\nu(y)\,d\mu(f).
\]
We would like to evaluate this by reversing the order of integration and 
seeing what the contribution to the integral is at each point $y$. 
Although there may be points $y$ at which the pointwise entropy is small,
the following Fubini argument estimates the average contribution.

We begin with a version of the argument for maps from one manifold to another. If we were only 
studying the Lyapunov exponents on $\SL(2,\R)$, then this estimate would suffice.
\begin{prop}[Fubini Argument]\label{prop:fubini_argument}
Let $M$ and $N$ be $d$-dimensional compact Riemannian manifolds of unit volume. 
Let $\mu$ be a probability measure on $\Diff(M,N)$ and suppose for each $y\in N$ 
we have a representation of $\mu$ as:
\[
\mu=\int_{\xi\in \Xi_{y}} \mu_{\xi}\,dm_y(\xi),
\]
where for each $\xi$, $f^{-1}(y)$ is constant for $\mu_{\xi}$-a.e.\ $f$.

Suppose that there exists $\beta>0$ and a set $G\subseteq \Diff(M,N)$, 
of ``good'' diffeomorphisms with the following property.
For any $\xi$, if $\supp \mu_{\xi}\cap G\ne\emptyset$ then $h_{\PW}(y,\mu_\xi)\ge \beta$. 
Then if $\nu$ is any absolutely continuous measure on $M$, 
\[
\Phi_{\mu}(\nu)\ge \beta \mu(G). 
\]
\end{prop}

\begin{proof}
    We compute
    \begin{align*}
        &\Phi_\mu(\nu)=\int\int\log\left(\frac{df_*\nu}{\,d\mu*\nu\,}(y)\right)\frac{df_*\nu}
        {\,d\mu*\nu\,}(y)\,d\mu(f)\,d(\mu*\nu)(y)\\
        &=\int\int\int\log\left(\frac{df_*\nu}{\,d\mu*\nu\,}(y)\right)\frac{df_*\nu}{\,d\mu*\nu\,}(y)
        \,d\mu_{\xi}(f)\,dm_y(\xi)\,d(\mu*\nu)(y)\\
       &=\int\int\int\left(
       \log\left(\frac{df_*\nu}{\,d\mu_{\xi}*\nu\,}(y)\right)
       +\log\left(\frac{d\mu_{\xi}*\nu}{\,d\mu*\nu\,}(y)\right)
       \right)\frac{df_*\nu}{\,d\mu*\nu\,}(y)
        \,d\mu_{\xi}(f)\,dm_y(\xi)\,d(\mu*\nu)(y)\\ 
        &=\circled{1}+\circled{2},
    \end{align*}
    where
    \begin{align*}
        \circled{1}&=
       \int\int\int\log\left(\frac{df_*\nu}{\,d\mu_{\xi}*\nu\,}(y)\right)
       \frac{df_*\nu}{\,d\mu*\nu\,}(y)
        \,d\mu_{\xi}(f)\,dm_y(\xi)\,d(\mu*\nu)(y)\text{;\quad and}\\
        \circled{2}&=
        \int\int\int
       \log\left(\frac{d\mu_{\xi}*\nu}{\,d\mu*\nu\,}(y)\right)
       \frac{df_*\nu}{\,d\mu*\nu\,}(y)
        \,d\mu_{\xi}(f)\,dm_y(\xi)\,d(\mu*\nu)(y).
    \end{align*}
Since the first term of \circled{2} does not depend on $f$, we have
$$
\circled{2}=\int\left(\int
       \log\left(\frac{d\mu_{\xi}*\nu}{\,d\mu*\nu\,}(y)\right)
       \frac{\,d\mu_{\xi}*\nu\,}{d\mu*\nu}(y)
        \,dm_y(\xi)\right)\,d(\mu*\nu)(y).
$$
Since $\int \frac{d\mu_{\xi}*\nu}{d\mu*\nu}(y)\,dm_y(\xi)=1$, 
by Jensen's inequality, the inner integral is non-negative, so $\Phi_\mu(\nu)\ge\circled 1$.
We rewrite $\circled 1$ as
\begin{align*}
&\circled 1=\int\int\int \log\left(\frac{df_*\nu}{\,d\mu_{\xi}*\nu\,}(y)\right)
       \frac{df_*\nu}{\,d\mu_{\xi}*\nu\,}(y)\frac{\,d\mu_{\xi}*\nu\,}{d\mu*\nu}(y)
        \,d\mu_{\xi}(f)\,dm_y(\xi)\,d(\mu*\nu)(y)\\
&=\int\int\left(\int \log\left(\frac{df_*\nu}{\,d\mu_{\xi}*\nu\,}(y)\right)
       \frac{df_*\nu}{\,d\mu_{\xi}*\nu\,}(y)\,d\mu_{\xi}(f)\right)
       \frac{\,d\mu_{\xi}*\nu\,}{d\mu*\nu}(y)
        \,dm_y(\xi)\,d(\mu*\nu)(y)\\
&=\int\int h_\PW(y,\mu_{\xi})
        \frac{\,d\mu_{\xi}*\nu\,}{d\mu*\nu}(y)\,dm_y(\xi)\,d(\mu*\nu)(y).
\end{align*}
We say a fiber, $\xi$, is \emph{good} if $\supp\mu_\xi\cap G\ne\emptyset$. 
By hypothesis, $h_\PW(y,\mu_\xi)\ge \beta\mathbf 1_\text{$\xi$ good}$.
So we have
\begin{equation}\label{eq:lowerbd}
\begin{split}
    \circled 1&\ge\beta 
    \int\int \mathbf 1_\text{$\xi$ good}\frac{\,d\mu_{\xi}*\nu\,}{d\mu*\nu}(y)
        \,dm_y(\xi)\,d(\mu*\nu)(y)\\
&=\beta \int\int \mathbf 1_\text{$\xi$ good}\int \frac{df_*\nu}{d\mu*\nu}(y)
\,d\mu_{\xi}(f)\,dm_y(\xi)\,d(\mu*\nu)(y)\\
&\ge \beta \int\int \int \mathbf 1_G(f)\frac{df_*\nu}{d\mu*\nu}(y)
\,d\mu_{\xi}(f)\,dm_y(\xi)\,d(\mu*\nu)(y)\\
&= \beta \int \int \mathbf 1_G(f)\frac{df_*\nu}{d\mu*\nu}(y)
\,d\mu(f)\,d(\mu*\nu)(y)\\
&=\beta \int \mathbf 1_G(f)\left(\int 1\,
\frac{df_*\nu}{d\mu*\nu}(y)\,d(\mu*\nu)(y)\right)\,d\mu(f)\\
&=\beta \int \mathbf 1_G(f)\left(\int 1\,\,df_*\nu(y)\right)\,d\mu(f),
\end{split}
\end{equation}
where the third line holds since $\mathbf 1_G(f)\le \mathbf 1_{\xi\text{ good}}$ for $\mu_\xi$-a.e. $f$.
Since the inner integral is $1$, we obtain
\[
\Phi_\mu(\nu)\ge \circled 1\ge \beta \mu(G).
\]
\end{proof}

A small generalization of the above argument gives a version that applies to random bundle maps. 
The important thing in the following is to note that the conclusion is now averaged. 

\begin{prop}[Fubini Argument for Bundles]\label{prop:fubini_argument_for_bundles}
 Suppose that $\pi\colon E\to M$ is a fiber bundle with smooth fibers and that 
 $\mu$ is a measure supported on the set of diffeomorphisms $\Diff(E,M)$ of $E$ 
 that preserve the fibering over $M$.

Let $\mc F_1$ be a fiber over $M$. 
For each fiber $\mc{F}_2$, let $\mu_{\mc{F}_2}$ be $\mu$ conditioned 
on the event that $f ({\mc{F}}_1)=\mc{F}_2$. 
We suppose we have a further decomposition of these measures:
for a fiber  $\mc F_2$ and for 
each $y\in \mc{F}_2$ we have a representation of $\mu_{\mc F_2}$ as:
\[
\mu_{\mc{F}_2}=\int_{\Xi_{y}^{\mc{F}_2}} \mu_{\xi}\,dm_y(\xi),
\]
where each $\mu_{\xi}$ satisfies that for all $f\in \supp\mu_{\xi}$ that $f^{-1}(y)$ is constant. 

Suppose that there exists $\beta>0$ and a set $G\subseteq \Diff(E,M)$, 
of ``good'' diffeomorphisms with the following property:
if $\xi\in\Xi_y^{\mc F_2}$ and $\supp\mu_\xi\cap G\ne\emptyset$,  
then $h_\PW(y, \mu_\xi)\ge \beta$. 
Then if $\nu$ is any absolutely continuous measure on a fiber $\mc F_1$, we have
\[
\int\Phi_{\mu_{\mc F_2}}(\nu)\,d(\mu*\delta_{\mc F_1})(\mc F_2)\ge\beta\mu(G).
\]
In particular, if the minimizer of the fiberwise Furstenberg entropy is absolutely 
continuous, then $\mu$ has uniformly positive fiber-averaged push-forward entropy.
\end{prop}

\begin{proof}
The argument is a slight elaboration of the argument in Proposition 
\ref{prop:fubini_argument}. Given a pair of fibers $\mc F_1$ and $\mc F_2$
and an absolutely continuous probability measure $\nu$ on $\mc F_1$,
from \eqref{eq:lowerbd},  we can extract the estimate:
\[
\Phi_{\mu_{\mc{F}_2}}(\nu)\ge 
\beta \int \mathbf 1_G(f)\left(\int 1\,\,
df_*\nu(y)\right)\,d\mu_{\mc{F}_2}(f)=
\beta \int {\mathbf 1}_G(f)\,d\mu_{\mc{F}_2}(f).
\]
Note that this estimate is not yet enough to conclude anything about the 
fiberwise entropy for any particular 
fiber: we don't know how much of $\mu_{\mc{F}_1,\mc{F}_2}$ is comprised of ``good" $f$. 
Now we can integrate 
this over all images $\mc{F}_2$ of $\mc{F}_1$. 
\[
\int \Phi_{\mu_{\mc{F}_2}}(\nu)\,d(\mu*\delta_{\mc{F}_1})(\mc{F}_2)\ge 
\beta\int_{M} \int \mathbf 1_G(f)\,d\mu_{\mc{F}_2}(f)\,
d(\mu*\delta_{\mc{F}_1})(\mc F_2)= \beta\mu(G).
\]
As long as the minimizers of the Furstenberg entropy are actually absolutely continuous, 
the final conclusion of the Proposition follows.
\end{proof}

\subsection{Mean miniflag volume distortion implies singular value gaps}
We are interested in obtaining the ratio of a pair of consecutive singular 
values of a product of matrices. 
The following proposition relates the logarithm of the ratio of the $k$th and 
$(k+1)$st singular values of a matrix
to an average value of the logarithm of the 1st and 2nd singular values 
for a corresponding fibered action.
For a matrix $A$, we write $\Svals_i^j(A)$ for the product 
$\sval_i(A)\sval_{i+1}(A)\cdots \sval_j(A)$ where
these are the $i$th through $j$th singular values, listed in non-increasing order.

Given a nested pair of subspaces $U\subseteq V$, we denote the orthogonal
complement of $U$ in $V$ by $V\ominus U$. Given a matrix $B$, 
we define the action of $B$ on the orthogonal complement of $U$ in $V$, denoted 
$B|_{V\perp U}$, by $\proj_{B(V)\ominus B(U)}B\proj_{V\ominus U}$, where $\proj_V$ 
denotes the orthogonal projection onto $V$.
For a pair of matrices $B_1$, $B_2$, one can check that
$(B_2B_1)|_{V\perp U}=(B_2)_{V_1\perp U_1}\cdot (B_1)_{V\perp U}$, where $U_1=B_1(U)$ and $V_1=B_1(V)$.
If $U$ is a $k$-dimensional subspace of an $\ell$-dimensional subspace $V$, we have the equality. 
\begin{equation}\label{eqn:singular_values_subspaces}
\Svals_1^{\ell}(B|_V)\coloneqq \Svals_1^\ell(B\proj_V)=\Svals_1^k(B|_U)\Svals_1^{\ell-k}(B|_{V\perp U}).
\end{equation}
This can be checked by writing the transformation $B$ with respect to an orthonormal basis 
with $U\oplus U^{\perp}$ on the domain and $B(U)\oplus B(U)^{\perp}$ on the codomain. The 
resulting matrix is block triangular and the singular values of the diagonal 
blocks are exactly those in the formula.

Given a partial flag
$\PF(i_1,i_2,\ldots,i_j)$, let $\lambda_{i_1,\ldots,i_j}$ be the normalized 
Haar measure, invariant under the
action of the orthogonal group, and recall that if $(i_{m_1},\ldots,i_{m_k})$ is a subsequence
of $(i_1,\ldots,i_j)$, then $\lambda_{i_{m_1},\ldots,i_{m_k}}$ is the 
marginal of $\lambda_{i_1,\ldots,i_j}$
under projection onto those coordinates. 

\begin{prop}\label{prop:gapest}
    For each $d\in\N$, there exists a constant $C_d$, such that for all $B\in\GL(d,\R)$,
    and all $1\le k<d$,
$$
\left|\int \log \frac {\sval_1(B|_{V_{k+1}\perp V_{k-1}})}{\sval_2(B|_{V_{k+1}\perp V_{k-1}})}\,
d\lambda_{k-1,k+1}(V_{k-1},V_{k+1})
-\log\frac {\sval_k(B)}{\sval_{k+1}(B)}\right|\le C_d.
$$
\end{prop}

For $V_\ell\in\grass_{\ell}$, let $\alpha_\ell(V_\ell)=S_1^{\ell}(\proj_{E_\ell}\proj_{V_\ell})$, 
where $E_\ell=\lin(e_1,\ldots,e_\ell)$.
This is a measure of the ``component'' of $V_\ell$ in the $E_\ell$ direction.

\begin{lem}\label{lem:gapcomp}
    Let $D=\diag(b_1,\ldots,b_d)$ and let $1\le k<d$ and let $(V_{k-1},V_k,V_{k+1})\in \PF(k-1,k,k+1)$. 
    Writing $\alpha_i$ for $\alpha_i(V_i)$, We have
$$\alpha_k^2\alpha_{k-1}\cdot\frac {b_k}{b_{k+1}}
\le \frac{\sval_1(D|_{V_{k+1}\perp V_{k-1}})}{\sval_2(D|_{V_{k+1}\perp V_{k-1}})}
\le \frac{1}{\alpha_{k-1}^2\alpha_{k+1}}\cdot \frac {b_k}{b_{k+1}}
$$
\end{lem}

\begin{proof}
For each $\ell$, we have the following inequalities:
\begin{equation}\label{eq:SVbounds}
    \alpha_\ell b_1\cdots b_\ell\le \Svals_1^\ell(D|_{V_\ell})\le b_1\cdots b_\ell,\\
\end{equation}
where the upper bound is by standard properties of singular values and the left hand side 
follows from 
\begin{align*}
    \Svals_1^\ell(D|_{V_\ell})&=\Svals_1^\ell(D\proj_{V_\ell})\\
    &\ge \Svals_1^\ell(\proj_{E_\ell}D\proj_{V_\ell})\\
    &=\Svals_1^\ell(\proj_{E_\ell}D\proj_{E_\ell}\proj_{E_\ell}\proj_{V_\ell})\\
    &=\Svals_1^\ell(\proj_{E_\ell}D\proj_{E_\ell})\Svals_1^\ell(\proj_{E_\ell}\proj_{V_\ell})\\
   &= \alpha_\ell\cdot b_1\cdots b_\ell.
\end{align*}
The second line follows by standard properties of singular values. The third line follows since
$\proj_{E_\ell}D=\proj_{E_\ell}D\proj_{E_\ell}^2$ and the fourth line follows by 
multiplicativity of determinants
since $\proj_{E_\ell}D\proj_{E_\ell}$ and $\proj_{E_\ell}\proj_{V_\ell}$ are rank $\ell$ matrices
where the domain of the first matches the range of the second. 
From earlier, if $W$ is any $k$-dimensional subspace containing $V_{k-1}$,
\begin{align}
    \Svals_1^{k}(D|_{W})&=\Svals_1^{k-1}(D|_{V_{k-1}})\sval_1(D|_{W\perp V_{k-1}});
    \text{\quad and}\label{eq:s1bound}
    \\
    \Svals_1^{k+1}(D|_{V_{k+1}})&=\Svals_1^{k-1}(D|_{V_{k-1}})
    \Svals_1^2(D|_{V_{k+1}\perp V_{k-1}}),\label{eq:s12bound}.
\end{align}

We then have
$$
\sval_1(D|_{W\perp V_{k-1}})=\frac{\Svals_1^k(D|_{W})}{\Svals_1^{k-1}(D|_{V_{k-1}})}
\le \frac{\Svals_1^k(D)}{\alpha_{k-1} b_1\cdots b_{k-1}}
=\frac{b_k}{\alpha_{k-1}}.
$$

Since $W$ was arbitrary, $\sval_1(D|_{\R^d\perp V_{k-1}})\le \frac{b_k}{\alpha_{k-1}}$, so that in 
particular
we see 
$$
\sval_1(D|_{V_{k+1}\perp V_{k-1}})\le \frac {b_k}{\alpha_{k-1}}.
$$
Now, taking $W$ to be $V_k$ in \eqref{eq:s1bound}, we have
$$
\sval_1(D|_{V_{k+1}\perp V_{k-1}})
\ge \sval_1(D|_{V_k\perp V_{k-1}})=\frac {\Svals_1^k(D|_{V_{k}})}{\Svals_1^{k-1}(D|_{V_{k-1}})}
\ge \frac{\alpha_k b_1\cdots b_k}{b_1\cdots b_{k-1}}
=\alpha_kb_k.
$$
Combining the bounds, we have
\begin{equation}\label{eq:s1bounds}
\alpha_k b_k\le \sval_1(D|_{V_{k+1}\perp V_{k-1}})\le \frac{b_k}{\alpha_{k-1}}.
\end{equation}

Since $\Svals_1^2(D|_{V_{k+1}\perp V_{k-1}})=\Svals_1^{k+1}(D|_{V_{k+1}})/\Svals_1^{k-1}(D|_{V_{k-1}})$,
applying the bounds in \eqref{eq:SVbounds} gives
\begin{equation}\label{eq:s12bounds}
    \alpha_{k+1}b_kb_{k+1}\le \Svals_1^2(B_{V_{k+1}\perp V_{k-1}})\le \frac{b_kb_{k+1}}{\alpha_{k-1}}.
\end{equation}

Dividing the square of \eqref{eq:s1bounds} by \eqref{eq:s12bounds} gives
the claimed result.
\end{proof}

\begin{lem}\label{lem:integrable}
    Let $\alpha_k\colon \grass_k\to\R$ be defined by 
    $\alpha_k(V_k)=\Svals_1^k(\proj_{E_k}\proj_{V_k})$ as above. 
    For all $1\le k<d$, $\log\alpha_k$ is an integrable function of 
    $V_k$ with respect to $\lambda_k$. 
\end{lem}

\begin{proof}
    Since $\alpha_k(V_k)\le 1$ for all $V_k\in\grass_k$ 
    (so that $\log\alpha_k$ takes non-positive values),
    we have the equality
\begin{align*}
    \int|\log\alpha_k(V_k)|\,d\lambda_k(V_k)&=
    \int_0^\infty \lambda_k(\{V_k\colon |\log\alpha_k(V_k)|\ge t\}\,dt\\
    &=\int_0^\infty \lambda_k(\{V_k\colon \alpha_k(V_k)\le e^{-t}\})\,d\!t.
\end{align*}
Recall that we may obtain a uniform random element of $V_k$ by sampling 
vectors $\mathbf v_1,\ldots,\mathbf v_k$ independently with independent standard normal entries
and setting $V_k=\lin(\mathbf v_1,\ldots,\mathbf v_k)$. Let 
$\mathbf u_i=\proj_{\lin(\mathbf v_1,\ldots,\mathbf v_{i-1})^\perp}(\mathbf v_i)$ and 
set $\hat{\mathbf u}_i=\mathbf u_i/\|\mathbf u_i\|$.
That is, $\hat{\mathbf u}_1,\ldots,\hat{\mathbf u}_k$ are the Gram-Schmidt orthonormalization
of $\mathbf v_1,\ldots,\mathbf v_k$. 
We then have 
$$
\alpha_k(V_k)=\prod_{j=1}^k
d\big(\proj_{E_k}\hat{\mathbf u}_j,\proj_{E_k}\lin(\mathbf v_1,\ldots,\mathbf v_{j-1})\big).
$$
In particular, if $\alpha_k(V_k)\le e^{-t}$, then for some $j\le k$, 
$d\big(\proj_{E_k}\hat{\mathbf u}_j,\proj_{E_k}\lin(\mathbf v_1,\ldots,
\mathbf v_{j-1})\big)\le e^{-t/k}$. 
That is
\begin{align*}
\lambda_k(\{V_k\colon \alpha_k(V_k)\le e^{-t}\})&\le
\sum_{j=1}^k \mathbb P\Big(d\big(\proj_{E_k}\hat{\mathbf u}_j,
\proj_{E_k}\lin(\mathbf v_1,\ldots,\mathbf v_{j-1})\big)\le e^{-t/k}\Big)\\
&\le k\,\mathbb P\Big(d\big(\proj_{E_k}\hat{\mathbf u}_k,
\proj_{E_k}\lin(\mathbf v_1,\ldots,\mathbf v_{k-1})\big)\le e^{-t/k}\Big)\\
&=k\,\mathbb P\Big(d\big(\proj_{E_k}\mathbf u_k,
\proj_{E_k}\lin(\mathbf v_1,\ldots,\mathbf v_{k-1})\big)\le e^{-t/k}\|\mathbf u_k\|\Big)\\
&=k\,\mathbb P\Big(d\big(\proj_{E_k}\mathbf v_k,
\proj_{E_k}\lin(\mathbf v_1,\ldots,\mathbf v_{k-1})\big)\le e^{-t/k}\|\mathbf u_k\|\Big)\\
&\le k\,\mathbb P\Big(d\big(\proj_{E_k}\mathbf v_k,
\proj_{E_k}\lin(\mathbf v_1,\ldots,\mathbf v_{k-1})\big)\le e^{-t/k}\|\mathbf v_k\|\Big)\\
\end{align*}

We observe
\begin{align*}
&\mathbb P\Big(d\big(\proj_{E_k}\mathbf v_k,
\proj_{E_k}\lin(\mathbf v_1,\ldots,\mathbf v_{k-1})\big)\le e^{-t/k}\|\mathbf v_k\|\Big)\\
&\le
\mathbb P\Big(\|\mathbf v_k\|\ge t\Big)+\mathbb P\Big(d\big(\proj_{E_k}\mathbf v_k,
\proj_{E_k}\lin(\mathbf v_1,\ldots,\mathbf v_{k-1})\big)\le te^{-t/k}\Big).
\end{align*}

The first term is $O(e^{-t^2/(2d)})$: at least one of the entries has to have 
absolute value exceeding $t/\sqrt d$,
so this is an integrable function of $t$. To estimate the second term, 
let $W$ be an arbitrary $(k-1)$-dimensional 
subspace of $\R^k$ and let $\mathbf V$ be a random vector in $\R^k$ with 
independent standard normal entries. 
We want to estimate $\mathbb P(d(\mathbf V,W)\le te^{-t/k})$.
We have
\begin{align*}
\mathbb P(d(\mathbf V,W)\le te^{-t/k})
&\le
\mathbb P(\|\mathbf V\|\ge t)+\operatorname{Vol}(B_{te^{-t/k}}(W)\cap B_t(0))\\
&=O(e^{-t^2/(2k^2)})+O(t^ke^{-t/k}).
\end{align*}
Since this holds for arbitrary $W$, 
we deduce
$$
\mathbb P\Big(d\big(\proj_{E_k}\mathbf v_k,
\proj_{E_k}\lin(\mathbf v_1,\ldots,\mathbf v_{k-1})\big)\le te^{-t/k}\Big)=O(t^ke^{-t/k}).
$$
Hence
$$
\int_0^\infty \lambda_k(\{V_k\colon \alpha_k(V_k)\le e^{-t}\})\,d\!t<\infty,
$$
as required.
\end{proof}

\begin{proof}[Proof of Proposition \ref{prop:gapest}]
By the singular value decomposition, there exist orthogonal matrices matrices $O_1$ and $O_2$ 
and a diagonal matrix $D$ such that $B=O_1DO_2$. 
If $O$ is an orthogonal matrix and $V$ is a subspace of $\R^d$, we first observe
$O\proj_V=\proj_{O(V)}O$.
If $A$ is any matrix and $U$ is a proper subspace of a subspace $V$,
then 
\begin{align*}
\sval_i((AO)_{V\perp U})&=\sval_i(\proj_{AO(V)\ominus AO(U)}AO\proj_{V\ominus U})\\
&=\sval_i(\proj_{AO(V)\ominus AO(U)} A\proj_{O(V)\ominus O(U)}O)\\
&=\sval_i(\proj_{AO(V)\ominus AO(U)} A\proj_{O(V)\ominus O(U)})\\
&=\sval_i(A|_{O(V)\perp O(U)}).
\end{align*}
Similarly, 
\begin{align*}
\sval_i((OA)_{V\perp U})&=\sval_i(\proj_{OA(V)\ominus OA(U)} OA\proj_{V\ominus U})\\
&=\sval_i(O\proj_{A(V)\ominus A(U)}A\proj V)\\
&=\sval_i(OA_{V\perp U})\\
&=\sval_i(A_{V\perp U}).
\end{align*}
Hence we have $\sval_i(B|_{V_{k+1}\perp V_{k-1}})=\sval_i((DO_2)|_{V_{k+1}\perp V_{k-1}})=
\sval_i(D_{O_2(V_{k+1})\perp O_2(V_{k-1})})$. Additionally, $\sval_i(B)=\sval_i(D)$. 
Since $\lambda_{k-1,k+1}$ is invariant under the left action of the orthogonal group,
we see
\begin{equation}\label{eq:D=B}
\begin{split}
&\left|\int \log \frac {\sval_1(B|_{V_{k+1}\perp V_{k-1}})}{\sval_2(B|_{V_{k+1}\perp V_{k-1}})}\,
d\lambda_{k-1,k+1}(V_{k-1},V_{k+1})
-\log\frac {\sval_k(B)}{\sval_{k+1}(B)}\right|\\
=&
\left|\int \log \frac {\sval_1(D|_{V_{k+1}\perp V_{k-1}})}{\sval_2(D|_{V_{k+1}\perp V_{k-1}})}\,
d\lambda_{k-1,k+1}(V_{k-1},V_{k+1})
-\log\frac {\sval_k(D)}{\sval_{k+1}(D)}\right|.
\end{split}
\end{equation}

By Lemma \ref{lem:gapcomp},
we have for
$\lambda_{k-1,k,k+1}$-a.e. $(V_{k-1},V_k,V_{k+1})$, 
$$
2\log \alpha_k+\log \alpha_{k-1}\le
\log\frac{\sval_1(D|_{V_{k+1}\perp V_{k-1}})}{\sval_2(D|_{V_{k+1}\perp V_{k-1}})}-
\log\frac {\sval_k(D)}{\sval_{k+1}(D)}\le -2\log\alpha_{k-1}-\log\alpha_{k+1}.
$$
Applying \eqref{eq:D=B} and integrating the inequality with respect to $\lambda_{k-1,k,k+1}$,
we obtain
\begin{align*}
&\int\left|\log\frac{\sval_1(B|_{V_{k+1}\perp V_{k-1}})}{\sval_2(B|_{V_{k+1}\perp V_{k-1}})}-
\log\frac {\sval_k(B)}{\sval_{k+1}(B)}\right|\,d\lambda_{k-1,k,k+1}(V_{k-1},V_k,V_{k+1})\\
\le &\int\left(2|\log\alpha_{k-1}|+2|\log\alpha_k|+|\log\alpha_{k+1}|\right)\,
d\lambda_{k-1,k,k+1}(V_{k-1},V_k,V_{k+1}).
\end{align*}
Since $\lambda_{k-1}$, $\lambda_k$, $\lambda_{k+1}$ and $\lambda_{k-1,k+1}$ are marginals of 
$\lambda_{k-1,k,k+1}$, applying Lemma \ref{lem:integrable}, we obtain the desired conclusion, 
observing that the right side does not depend on $B$.
\end{proof}

\begin{rem}
The estimate in Proposition \ref{prop:gapest} is closely related to certain estimates in the theory 
of Dedieu-Shub measures \cite{dedieu2003shub,burns1999recent}. For example, 
\cite{rivin2005mean} shows that 
\[
\int_{O(n)} \rho(AO)\,d\text{Haar}(O) \ge C_d \|A\|,
\]
where $O_n$ is the orthogonal group and $\rho(AO)$ is the spectral radius of $AO$. 
See also \cite{armentano2024random}.
\end{rem}
\subsection{Mean miniflag fiber distortion implies mean singular value gaps} 

The following lemma provides the necessary connection between the entropy on miniflags
and the gaps between singular values. 

\begin{prop}\label{prop:miniflag_distortion_implies_singular_value_gaps}
Let $m$ be the Haar measure on $\PF(k-1,k+1)$. 
Let $\mu$ be a measure on $\GL(d,\R)$. Then:
\[
\E_{\mu}\left[\log \sval_{k}(A)-\log\sval_{k+1}(A)\right]\ge 
\int_{\PF(k-1,k+1)} \EPIA(\mc F,\mu)\,dm(\mc F)-C_d,
\]
where $C_d$ is the constant from Proposition \ref{prop:gapest}.
In other words, the mean fiber-averaged entropy gives a lower bound on the mean gap 
between singular values.
\end{prop}

\begin{proof} 
Integrating Proposition \ref{prop:gapest} over $\mu$ gives
\begin{align*}
&\E_{\mu}\left[\log \sval_{k}(A)-\log\sval_{k+1}(A)\right]\\
&\ge \E_{\mu}\left[\int \log 
\frac {\sval_1(A|_{V_{k+1}\perp V_{k-1}})}{\sval_2(A|_{V_{k+1}\perp V_{k-1}})}\,
dm(\mc F_1)\right]-C_d\\
&=\int\int \log \frac {\sval_1(A|_{V_{k+1}\perp V_{k-1}})}{\sval_2(A|_{V_{k+1}\perp V_{k-1}})}
\,d\mu(A)\,dm(\mc F_1)-C_d\\
&=\int\int\log (J(A|_{V_{k+1}\perp V_{k-1}})^{-1})_\text{max}\,d\mu(A)\,dm(\mc F_1)-C_d\\
&=\int\int\int\log (J(A|_{V_{k+1}\perp V_{k-1}})^{-1})_\text{max}\,d\mu_{\mc F_1,\mc F_2}(A)
\,d\mu_{\mc F_1}(\mc F_2)\,dm(\mc F_1)-C_d\\
&\ge \int\int \EP(\mc F_1,\mu_{\mc F_1,\mc F_2})\,d\mu_{\mc F_1}(\mc F_2)\,dm(\mc F_1)-C_d\\
&=\int\EPIA(\mc F_1,\mu)\,dm(\mc F_1)-C_d,
\end{align*}
where $\mc F_1$ denotes the pair $(V_{k-1},V_{k+1})$ and we used Lemma 
\ref{lem:entropy_volume_distortion}
in the sixth line. 
\end{proof}

\section{Entropy of the induced random bundle map on miniflags}\label{sec:pointwise_entropy_est}

Next we will show that a perturbation of size $\epsilon$ gives rise to 
pointwise entropy that is quadratic in $\epsilon$. 

Before we begin the proof, we outline the approach. We recall that we are considering the action on certain partial flags by 
matrices of the form $A+\xi$ where $A$ is fixed and $\xi$ is a random perturbation.
Let $(V_{k-1},V_{k+1})\in \PF(k-1,k+1)$, let $\mc F_0=\mc F(V_{k-1},V_{k+1})$,
and let $x=(V_{k-1},V_k,V_{k+1})\in \PF(k-1,k,k+1)$ be a point in the corresponding
miniflag. Note that matrices act naturally on partial flags.
We then form an $x$-dependent partition of the $\xi\in \R^{d\times d}$,
where $\xi\sim\xi'$ if $(A+\xi)(x)=(A+\xi')(x)$.
Call this partition $\mc{Q}$. For $Q\in \mc Q$, $(A+\xi)(V_{k-1},V_{k+1})$
is a constant element of $\PF(k-1,k+1)$ for all $\xi\in Q$. We consider the action of
the $(A+\xi)_{\xi\in Q}$ on $\mc F(V_{k-1},V_{k+1})$ and look, in particular, at the 
variability of the Jacobian of the action at the point $x$ in the miniflag. 

The most important point in the calculation is the following: there are some images of $x$ 
that are quite rare and hence do not contribute much entropy for two reasons. First, they do not 
contribute much due to their rarity, but second there is also not much variation under the dynamics. 
For example, consider the case where $\xi$ is extremely close to the boundary of the support of 
$\mu_{\epsilon}$: it is hard to ``vary" $\xi$ much while keeping $(A+\xi)(x)$ 
fixed. For this reason, we will simplify our work by restricting to images $(\mc{F}_1, y)$ 
given by an element $A+\xi$ where $\xi$ is deep inside the support of $\mu_{\epsilon}$.
The concavity of entropy allows us to derive a lower bound for the entropy based on the
action of these deep $\xi$'s. 

\begin{prop}\label{prop:fiberwise_entropy_invertible_case}
Suppose that $\phi$ is a continuous density on 
$M_{d\times d}(\R)$ such that there exists $C>0$ with $\phi(A)\le C/\|A\|^{d^2+1}$
for all $A$. There exists $C_{\phi}$ 
such that the following holds. Suppose that $M\ge 1$ and $0\le \epsilon\le 1$ are given. 
Let $E$ be the absolutely continuously distributed matrix-valued random variable with density $\phi$.
If $A$ is any matrix with $\|A\|\le M$,
let the measure $\mu_{\epsilon}$ be the distribution of
\[
A+\epsilon E.
\]
Then $\mu_{\epsilon}$ has uniform fiberwise entropy gap of at least $C_{\phi}\epsilon^2/M^2$ when acting on 
the bundle of miniflags, $\PF(k-1,k,k+1)$ over $\PF(k-1,k+1)$
for any $1\le k<d$. 
\end{prop}

\begin{proof}
We assume that $\phi$ is positive at the zero matrix. This is without loss of generality as if 
$\phi$ is positive at some matrix $E_0$, then we may replace $\phi(\cdot)$ by $\phi(\cdot+E_0)$
and replace $A$ by $A+\epsilon E_0$, and arrive at a similar conclusion (with $C_\phi$
modified to take account of the possibly greater norm of $A+\epsilon E_0$). 
Let $K=\max\phi$.
For simplicity, we further assume 
that $\phi(E)\ge a>0$ whenever $\|E\|_\infty\le 1$, where $\|E\|_\infty$ denotes $\max_{i,j}|E_{ij}|$. 
Clearly, a bound of this type holds on some axis-aligned cube centered at the origin. If it were
a smaller cube, some minor adjustments would need to be made. 

Next, we will find a subset of the support of $\mu_{\epsilon}$ that is disjoint from the
non-invertible matrices and large enough that we may ``vary freely" within it.
Let $A$ be fixed with $\|A\|\le M$ and
let $O_1DO_2$ be a singular value decomposition of $A$.
Then there exists a diagonal matrix $D'$ with $s_d(D')\ge \frac\epsilon 2$ such that $\|D'-D\|\le
\frac\epsilon2$. Set $A'=O_1D'O_2$ so that $\|A'-A\|_\infty\le \|A'-A\|=\|D'-D\|\le\frac\epsilon 2$
and $s_d(A')=s_d(D')\ge\frac\epsilon 2$.
Now if $\|B-A'\|_\infty<\frac\epsilon{4 d}$, then $\|B-A'\|<\frac\epsilon 4$, so that 
by the Lipschitz property of singular values, $s_d(B)\ge s_d(A')-\frac\epsilon 4\ge \frac\epsilon 4$. 
Let $G_{\epsilon,1}=\{B\colon \|B-A'\|_\infty\le\frac\epsilon{4 d}\}$ and
let $G_{\epsilon,2}=\{B\colon \|B-A'\|_\infty\le \frac\epsilon{8 d}\}$.
Note that any line $\ell$ 
containing a point $B\in G_{\epsilon,2}$ satisfies that the length of the intersection 
$G_{\epsilon,1}\cap \ell$ is at least $\frac{\epsilon}{4 d}$. 
We see that $\mu_\epsilon(G_{\epsilon,1})\ge a(\frac 1{2d})^{d^2}$, as the space of $d\times d$ matrices has dimension $d^2$.
We suppress the $\epsilon$ subscript, so just write $\mu$ for the remainder
of the proof. 
By Lemma \ref{lem:convexity_entropy}, at the cost of introducing a factor
$a/(2d)^{d^2}$, we may replace $\mu$ with normalized Lebesgue measure on $G_{\epsilon,1}$.

We aim to apply Proposition \ref{prop:fubini_argument_for_bundles} to the bundle of miniflags.
Let $\mc F_1=\MF(V_{k-1},V_{k+1})$ be a miniflag. Let $v_1,\ldots,v_d$ 
be a (measurably chosen) orthonormal basis for $\R^d$ such that $\lin(v_1,\ldots,v_j)=V_j$
for $j=k-1,k+1$. 
In order to apply the proposition, we need to condition on the image
of $\mc F_1$, $\mc{F}_2$, giving the measure $\mu_{\mc{F}_2}$ and then subdivide further with a partition 
$\Xi_y$ that keeps the preimage of the point $y\in \MF(\mc F_2)$ constant.

To this end, we define five measurable partitions of $M_{d\times d}(\R)$, $\Xi_1,\ldots,\Xi_5$. 
The goal is to find partitions suitable for Proposition \ref{prop:fubini_argument_for_bundles}
where the densities of the fiber measures are easy to compute. It turns out
that these computations are simpler if the partitioning is done in stages rather than all 
at once. In particular, we repeatedly make use of the fact that if the elements of a partition 
into parallel hyperplanes have absolutely continuous fiber measures and the partition
is refined by a second partition into lower-dimensional parallel hyperplanes 
then, up to normalization, the fiber measures of the refined partition have the 
same density as the fibers of the coarser partition.

For $\mathbf w=(w_1,\ldots,w_{k-1},w_{k+2},\ldots,w_d)\in (\R^d)^{d-2}$, let
$W_{k-1}=\lin(w_1,\ldots,w_{k-1})$, and define
$$
\xi_1(\mathbf w)=\{A\colon Av_j=w_j\text{ for all $j\ne k,k+1$}\},
$$
and we let $\Xi_1=\{\xi_1(\mathbf w)\colon \mathbf w\in (\R^d)^{d-2}\}$. 
Since $\Xi_1$ is a decomposition of $M_{d\times d}(\R)$ into parallel hyperplanes, the fiber measures on the $\xi_1(\mathbf w)$
atoms have a density proportional to the original density $\phi$. The $\Xi_1$ fibers are $2d$-dimensional.

Given $\mathbf w$ as above, for $b_k,b_{k+1}\in W_{k-1}$, define
$$
\xi_2(\mathbf w,b_k,b_{k+1})=\{A\in \xi_1(\mathbf w)\colon \Pi_{W_{k-1}}(Av_j)=b_j\text{ for $j=k,k+1$}\}.
$$
Here, as elsewhere, $\Pi_V$ denotes the orthogonal projection onto $V$.
We then let $\Xi_2$ be $\{\xi_2(\mathbf w,b_k,b_{k+1})\colon \mathbf w\in (\R^d)^{d-2},\ b_k,b_{k+1}\in W_{k-1}\}$.
As above, the fiber measures on the fibers of $\Xi_2$ have density $\phi$. The fibers of $\Xi_2$ are 
$(2d-2k+2)$-dimensional.
Note that although $\Xi_1$ is a parallel decomposition of $M_{d\times d}$ and $\Xi_2$ is a parallel 
decomposition of $\Xi_1$, $\Xi_2$ is not, itself, a parallel decomposition of $M_{d\times d}$. This
is the reason that we choose to use a sequence of partitions.

Next, for $U\in \Gr_2(W_{k-1}^\perp)$, define
$$
\xi_3(\mathbf w,b_k,b_{k+1},U)=\{A\in \xi_2(\mathbf w,b_k,b_{k+1})\colon \Pi_{W_{k-1}^\perp}(Av_j)\in U\text{ for $j=k,k+1$}\}.
$$
Let $\Xi_3$ be $\{\xi_3(\mathbf w,b_k,b_{k+1},U)\colon \mathbf w\in (\R^d)^{d-2},\ b_k,b_{k+1}\in W_{k-1},
U\in\Gr_2(W_{k-1}^\perp)\}$.
The density on $\xi_3(\mathbf w,b_k,b_{k+1},U)$ is proportional to
\begin{equation}\label{eq:xi3fibredensity}
\begin{split}
\phi_3(A)&=
\sin\angle(\Pi_U(Av_k),\Pi_U(Av_{k+1}))\big)^{d-k-1}\mathbf 1_{G_{\epsilon,1}}(A)\\
&
=\|\Pi_UAv_k\wedge \Pi_UAv_{k+1}\|^{d-k-1}\mathbf 1_{G_{\epsilon,1}}(A).
\end{split}
\end{equation}
A proof of this is given, using the co-area formula, in Appendix \ref{app:coarea}.
The $\Xi_3$ fibers are 4-dimensional. 

So far, $\Xi_1$ specifies the image under $A$ of $v_1,\ldots,v_{k-1}$ and $v_{k+2},\ldots,v_d$; 
$\Xi_2$ additionally specifies the projection of $Av_k$ and $Av_{k+1}$ to $W_{k-1}$;
$\Xi_3$ further specifies the 2-dimensional subspace of $W_{k-1}^\perp$ that $\Pi_{W_{k-1}^\perp}Av_k$ and
$\Pi_{W_{k-1}^\perp}Av_{k+1}$ lie in. 
Given an element $\mc F_2=(W_{k-1},W_{k+1})\in\PF(k-1,k+1)$, we observe that the elements of 
$M_{d\times d}$ sending $\mc F_1$ to $\mc F_2$ are given by 
$$
\bigcup \xi_3(\mathbf w,b_k,b_{k+1},U),
$$
where the union is taken over $\mathbf w$ such that $\lin(w_1,\ldots,w_{k-1})=W_{k-1}$; $b_k,b_{k+1}\in W_{k-1}$; and
$\lin(W_{k-1},U)=W_{k+1}$. That is, the partition of $M_{d\times d}$ according to $A(\mc F_1)$
is coarser than $\Xi_3$. 

We now further partition an element $\xi_3(\mathbf w,b_k,b_{k+1},U)$ of $\Xi_3$.
Since these data will not change, we 
suppress the $\mathbf w,b_k,b_{k+1},U$ from the notation.
Let $W_{k+1}=\lin(W_{k-1},U)$ and let $(W_{k-1},W_k,W_{k+1})$ be an
element of $\MF(W_{k-1},W_{k+1})$ (playing the role of $y$ in Proposition \ref{prop:fubini_argument_for_bundles}). 
Let $y$ be a unit vector in $U\cap W_k$ 
and let $y^\perp$ be a unit vector in $U$ perpendicular to $y$.
Let
$$
\xi_4(\theta)=\{A\in \xi_3\colon \langle y^\perp,
\Pi_{W_{k-1}^\perp} A(\cos\theta \,v_k+\sin\theta\, v_{k+1})\rangle=0\},
$$
and let $\Xi_4=\{\xi_4(\theta)\colon \theta\in [0,\pi)\}$. 
That is, we partition $\xi_3$ according to the direction of the preimage 
of $y$ in $\lin(v_k,v_{k+1})$:
$\xi_4(\theta)$ consists of those $A$'s such that 
$\Pi_{W_{k-1}^\perp}A(\cos\theta\,v_k+\sin\theta\,v_{k+1})\in\lin(y)$.
These fibers are 3-dimensional.
The fiber measure density is proportional to
\begin{equation}\label{eq:xi4fibredensity}
\phi_4(A)=|\langle y^\perp,A(-\sin \theta \,v_k+\cos\theta\,v_{k+1})\rangle|\phi_3(A).
\end{equation}
Again, a proof is given in Appendix \ref{app:coarea}.

Finally, we partition $\xi_4(\theta)$ into 1-dimensional fibers
by specifying $\Pi_U(A(-\sin\theta\,v_k+\cos\theta\, v_{k+1}))$.
The partition $\Xi_5$ arising from this is again a parallel decomposition of $\xi_4(\theta)$ so 
the fiber density, $\phi_5(A)$, is proportional to $\phi_4(A)$. 

By definition, the $\Xi_5$ fiber consists of matrices $A$ such that
$Av_j=w_j$ for all $j\ne k,k+1$; $\Pi_{W_{k-1}}Av_j=b_j$ for $j=k,k+1$;
$\Pi_{W_{k-1}^\perp}A(\cos\theta v_j+\sin\theta v_{j+1})\in\lin(y)$
and $\Pi_{W_{k-1}^\perp}A(-\sin\theta v_j+\cos\theta v_{j+1})=z$ for some $z\in W_{k+1}\ominus W_{k-1}$.
For a cleaner description, define a new orthonormal basis by 
$v'_j=v_j$ for $j\ne k,k+1$, $v'_k=\cos\theta v_k+\sin\theta v_{k+1}$, 
$v'_{k+1}=-\sin\theta v_k+\cos\theta v_{k+1}$. Let $b'_k=\cos\theta b_k+\sin\theta b_{k+1}$
and $b'_{k+1}=-\sin\theta b_k+\cos\theta b_{k+1}$. 
With respect to this basis, $\xi$ is the set of $A_t$ where $A_t$ satisfies
\begin{equation}\label{eq:A_tdef}
\begin{split}
A_tv_j'&=w_j\text{ for $j\ne k,k+1$;}\\
A_tv_k'&=ty+b_k'\\
A_tv_{k+1}'&=z+b_{k+1}'.
\end{split}
\end{equation}
That is, $A_t=X+tY$ where $X=WV^T$, $Y=yv_k'^T$, $W$ is the matrix with columns $w_1,\ldots,w_{k-1},
b_k',z+b_{k+1}',w_{k+2},\ldots,w_d$, and $V^T$ is the matrix with rows $v'_1,\ldots,v'_d$.

To apply Proposition \ref{prop:fubini_argument_for_bundles}, 
it suffices to show that for each element of the partition $\Xi_5$ that intersects $G_{\epsilon,2}$ 
that there is a lower bound on the pointwise entropy of the conditional measure along that atom.
We now do this, making use of the expression for the density in the fibers that we have computed.
Let $\xi=\xi_5(\theta,z)$. 
Notice that if $b_k'=\beta_1w_1+\ldots +\beta_{k-1}w_{k-1}$, then 
$A_t(v_k'-(\beta_1v_1'+\ldots +\beta_{k-1}v_{k-1}'))=ty$
so that if $A_t\in G_{\epsilon,1}$ then $|t|\ge \frac\epsilon 4$ as $y$ is a unit vector. 
We also have $\|A_t\|\le M+1$ if $A_t\in G_{\epsilon,1}$ by construction of $G_{\epsilon,1}$. 
Since $\|A_t\|\ge t$, we see that $|t|\le M+1$ if $A_t\in G_{\epsilon,1}$. 
As observed above,
if $\xi$ intersects $G_{\epsilon,2}$,
then the set of $t$ for which $A_t\in G_{\epsilon,1}$ is an interval of length
at least $\frac{\epsilon}{4d}$.
Since $A_t=X+tY$, where $Y$ is a unit matrix, $\|A_t-A_{t'}\|=|t-t'|$ so 
by the definition of $G_{\epsilon,1}$, the interval has length at most $\frac\epsilon 2$. 
Letting $[p,q]$ be the interval of $t$'s 
for which $A_t\in G_{\epsilon,1}$ we have shown
\begin{equation}\label{eq:p,q_bounds}
\begin{split}
&\tfrac\epsilon 4\le p<q\le M+1;\\
&\tfrac\epsilon{4d}\le |p-q|\le \tfrac\epsilon 2.
\end{split}
\end{equation}
Recall the density on the $\xi_5$ fiber, $\phi(A_t)$ for $t\in [p,q]$ is 
proportional to 
$$
\|\Pi_U(A_tv_k)\wedge \Pi_U(A_tv_{k+1})\|^{d-k-1}|\langle y^\perp,A_tv'_{k+1}\rangle|\mathbf 1_{G_{\epsilon,1}}(A_t).
$$
Using \eqref{eq:A_tdef}, this is proportional to $\rho(t)=t^{d-k-1}\mathbf 1_{[p,q]}(t)$.

Using \eqref{eq:p,q_bounds}, we see that for all $p\le t,t'\le q$,
\begin{equation}\label{eq:density_ratio_bound}
\frac{\rho(t)}{\rho(t')}\le \left(\frac qp\right)^{d-k-1}\le 3^{d-k-1}.
\end{equation}
We now compute the Jacobian of the map $\bar A_t^{-1}$ induced by $A_t^{-1}$ from $\mc F(W_{k-1},W_{k+1})$
to $\mc F(V_{k-1},V_{k+1})$ at the point $y=(W_{k-1},W_k,W_{k+1})$.
The elements of the miniflags are then parameterized by parameters
$\sigma$ and $\tau$ corresponding to the direction $\cos\sigma v_k'+\sin\sigma v_{k+1}'$ in 
$\MF(\mc F_1)$ and $\cos\tau y+\sin\tau y^\perp$ in $\MF(\mc F_2)$. 
Let $z=ay+by^\perp$. We then have $\Pi_U A_t v_k'=ty$ and $\Pi_U A_t v_{k+1}'=ay+by^\perp$ for all $t$. 
Hence 
\[
\Pi_U A_t(v_k'+h v_{k+1}')=(t+ah)\left(y+\frac {bh}{t+ah}y^\perp\right).
\]
As $\lim_{h\to 0}\big(\frac{bh}{t+ah}/h)=b/t$, so we see that the
Jacobian of $\bar A_t$ at $v_k'$ is given by $b/t$.
Now consider the quantity:
\[
Z\coloneqq \Jac_y(\bar A_t^{-1}) / \int \Jac_y(\bar A_t^{-1})\,d\mu_\xi(A_t),
\]
which appears in the definition of the pointwise entropy. We see the $Z$ value 
corresponding to $t$ is $Z_t=(t/b)/\int_p^q (s/b)\rho(s)\,ds=t/t^*$, where
$t^*=\int_p^q s\rho(s)\,ds$.
Using \eqref{eq:density_ratio_bound}, we see the quantity 
$Z:=\Jac_y(\bar A_t^{-1}) / \int \Jac_y(\bar A_t^{-1})\,d\mu_\xi(A_t)$ appearing in 
the pointwise entropy is bounded above by $3^{d-k-1}$. 
That is, $Z$ takes values in $[\frac p{t^*},\frac q{t^*}]$ and has density $\rho_Z(z)=\rho(z/t^*)/t^*$.
Using Lemma \ref{lem:convexity_xlogx}, we obtain 
$h_\PW(y,\mu_\xi)\ge \tfrac16 \var Z$.
Since $\rho_Z(z)/\rho_Z(z')\le 3^{d-k-1}$ for all $z,z'$ in the range of $Z$, Lemma 
\ref{lem:variance_lower_bound} applies, giving $h_\PW(y,\mu_\xi)\ge \frac 16(p-q)^2/({t^*}^23^{2d-2k-2})$.
By \eqref{eq:p,q_bounds}, $t^*$ is bounded above  by $M+1$, so we find that: 
$$
h_\PW(y,\mu_\xi)\ge C\epsilon^2/M^2,
$$
where $C$ does not depend on $A$ or the particular fiber as long as it intersects $G_{\epsilon,2}$.
Applying Proposition \ref{prop:fubini_argument_for_bundles} and 
re-inserting the factor $a/(2d)^{d^2}$ from above from when we replaced $\mu$ by a measure
supported on a cube, the conclusion now follows.
\end{proof}

\section{Conclusion}\label{sec:conclusion}
We can now establish the main theorem of this paper.
This 
section is structured as follows. First, we obtain a consequence of the main result of 
\cite{gorodetski2023nonstationary}, 
which shows that almost surely the logarithmic singular values of
realizations of matrix products differ from their expectation by $o(n)$.
We then relate our $\GL(d,\R)$ matrices to $\SL(d,\R)$
matrices and control some moments required by \cite{gorodetski2023nonstationary}
before finishing the proof.

\subsection{Another consequence of Gorodetski-Kleptsyn}
We first record a consequence of the main theorem in \cite{gorodetski2023nonstationary}. 
We denote by $\SLpm(d,\R)$ the collection of $d\times d$ matrices of determinant $\pm 1$.
Given a matrix 
$A\in \SLpm(\R^d)$, we will say that $A$ acts on we write $\bigwedge^k\R^d$ by
\[
(\textstyle{\bigwedge^k}A)(v_1\wedge \cdots \wedge v_n)=(Av_1)\wedge\cdots\wedge(Av_n). 
\]
Similarly, given a miniflag $\mc F=\mc F(V_{k-1},V_{k+1})$, the action of $A$ on $\mc F$ is given
by the matrix $A|_{V_{k+1}\perp V_k}$.

\begin{lem}\label{lem:no_invariants_big_support}
Suppose that $\mu$ is a measure defined on $\SLpm(d,\R)$ that has a component that equivalent to 
volume on a neighborhood of $\Id\in \SLpm(d,\R)$. Then there do not exist any pair of probability 
measures $\nu_1$ and $\nu_2$ on $\mathbb{P}(\R^{d})$, such that $\mu$-a.s.~$A_*\nu_1=\nu_2$. 
Further, for $1\le k\le d-1$, there do not exist probability measures $\nu_1$ and $\nu_2$ on 
$\mathbb{P}(\bigwedge^k \R^d)$ such that $\mu$-almost surely $(\bigwedge^k A)_*\nu_1=\nu_2$. 

Further, for $1\le k\le d-1$, there does not exist a non-trivial pair of subspaces $V_1$ 
and $V_2$ of $\bigwedge^k \R^d$ such that $\mu$-a.s.~$(\bigwedge^k A)_*V_1=V_2$.
\end{lem}

\begin{proof}
We begin with the statement about the action on $\mathbb P(\R^d)$.
Suppose for a contradiction that $\nu_1$ and 
$\nu_2$ are two probability measures on $\mathbb P(\R^d)$ such that $A_*\nu_1=\nu_2$ for $\mu$-a.e. $A$.
From the assumption, it suffices to consider 
the case that $\mu$ is fully supported on an open symmetric generating set $\mc{N}$ containing 
$\Id\in \SLpm(d,\R)$; by symmetric we mean that if $A\in \mc{N}$ then $A^{-1}\in \mc{N}$. As 
preserving a measure is a closed condition, we see that for any such pair $A_*\nu_1=\nu_2$ for all 
$A\in \mc{N}$. In particular, as $\Id\in \mc{N}$, this implies $\nu_1=\nu_2$. 
As $\mc{N}$ is symmetric and $\SL(d,\R)$ 
is connected, it generates $\SL(d,\R)$ as a semigroup. In particular, this 
implies that such a measure is 
invariant under all of $\SL(d,\R)$, which is impossible because this easily implies that $\nu$ must be a 
$\delta$-mass but $\SL(d,\R)$ does not preserve any fixed lines $[v_{\delta}]\in \RP^{d-1}$ where 
$\delta$ is supported.

When we consider the induced action on $\bigwedge^k\R^d$, the push-forward of $\mu$ by the map 
$A\mapsto \bigwedge^k A$ is not fully supported in a neighborhood of $\Id\in \SL(\bigwedge^k \R^d)$. 
The same considerations as before reduce to the case of showing that there is no measure $\nu$ on 
$\mathbb{P}(\bigwedge^k\R^d)$ invariant under all the $(\bigwedge^k)_*\mu$. From 
\cite[Cor.~3.10]{abels1995semigroups}, as long as the eigenvalues of an element 
$g\in \SL(\R^d)$ are distinct in modulus, then the induced action of $\bigwedge^k g$ on 
$\bigwedge^k \R^d$ is proximal.
By proximal, we mean that the 
action of $\|\bigwedge^k g^n\|^{-1}g^n$ converges to a rank $1$-matrix. In particular, this 
implies that $\mu$ must be a $\delta$-mass sitting on some vector $v_{\delta}\in \bigwedge^k\R^d$. 
However, as $(\bigwedge^k)_*\mu$ generates $\bigwedge^k\SL(\R^d)$, we see that $v_{\delta}$ is invariant 
under all of $\bigwedge^k \SL(\R^d)$. But this is impossible: $\bigwedge^k\SL(\R^d)$ is an irreducible 
representation of $\SL(d,\R)$.\footnote{For this fact, first $\bigwedge^k \SL(d,\mathbb{C})$ is an 
irreducible representation on $\bigwedge^k \mathbb{C}^d$ \cite[p.~340]{knapp2002lie}. As this is 
the complexification of $\bigwedge^k \SL(d,\R)$, we must have that $\bigwedge^k \SL(d,\R)$ is 
irreducible as well, as otherwise its complexification would be reducible.}

The proof of the statement about subspaces is similar and just uses that the representations 
$\bigwedge^k \SL(d,\R)$ are irreducible.
\end{proof}

\begin{rem}
Lemma \ref{lem:no_invariants_big_support} is stated for measures whose support contains a neighborhood 
of the origin. However, note that if $\mu$ is a measure as in Lemma \ref{lem:no_invariants_big_support}, 
then the same conclusions follow for the measure $\delta_A*\mu$ for any matrix $A\in \SLpm(d,\R)$.
\end{rem}

\begin{prop}\label{prop:growth_in_expectations_implies_almost_sure}
Suppose that $\mc{K}$ is a compact subset of the space of probability 
measures on $\SLpm(d,\R)$ and that any 
$\mu\in \mc{K}$ has a component that is equivalent to volume on a neighborhood of $\Id\in \SLpm(d,\R)$. 
Suppose that we have the following moment condition: there exist $\gamma,C$, such that for all 
$\mu\in \mc{K}$, $\int \|A\|^{\gamma}\,d\mu(A)<C$. 
Suppose that $(\mu_n)_{n\in \N}$ is a sequence of probability measures in $\mc{K}$. Let $T_n$ 
denote the (nonstationary) random product of the first $n$ matrices, so that $T_n$ is 
distributed according to $\mu^{*n}=\mu_n*\cdots*\mu_1$. Let
\[
L_{k,n}=\E\left[\log \prod_{j=1}^k\sval_k(T_n)\right]=\E\left[\log \|\textstyle\bigwedge^k T_n\|\right].
\]
Then for each $k$,
\[
\lim_{n\to \infty} \frac{1}{n} \left(\log\|\textstyle\bigwedge^k T_n\|-L_{k,n}\right)=0. 
\]
\end{prop}
\begin{proof}
We apply \cite[Thm.~1.1]{gorodetski2023nonstationary} to the induced random 
product on the exterior power. 
Let $\mu_{k,n}$ be the push-forward of $\mu_n$ from $\SLpm(d,\R)$ to $\SLpm(\bigwedge^k\R^d)$. 
We point out that while \cite{gorodetski2023nonstationary} is stated for actions of $\SL(d,\R)$,
it applies verbatim to actions of $\SLpm(d,\R)$.

We check that the hypotheses of that theorem hold for the sequence of measures $\mu_{k,n}$. There are 
three hypotheses to check: The finite moment condition, the measures condition, and the subspaces 
condition. The latter two conditions are verified due to Lemma \ref{lem:no_invariants_big_support}. 
The moments condition is verified because $\|\bigwedge^k A\|\le \|A\|^{k}$.

Thus we may apply \cite[Thm.~1.1]{gorodetski2023nonstationary} to the measures $\mu_{k,n}$. 
The conclusion is now immediate from that theorem because 
$\|\bigwedge^k T_n\|=\prod_{j=1}^k \sval_j(T_n)$. 
\end{proof}

\subsection{Hypotheses on the noise and moments}
The results of \cite{gorodetski2023nonstationary} are phrased for matrices in $\SL(d,\R)$, 
whereas here we work with matrices that are almost surely in $\GL(d,\R)$. 
As such, we need to 
relate matrices in $\GL(d,\R)$ to matrices in $\SL^\pm(d,\R)$ and check that the appropriate moment 
estimate holds. Given a matrix $A\in \GL(d,\R)$, we write $\hat A$ for the matrix
$A/|\det A|^{1/d}$. We call this the \emph{normalization} of $A$.
We observe that it has the same projective action as $A$ on
any miniflag. 

\begin{definition}\label{defn:uniform_norm_moment}
Given a collection $\mc K$ of probability measures on $M_{d\times d}(\R)$ it
has finite \emph{uniform $\gamma$-norm moment} if
$$
\sup_{\mu\in \mc K}\int \|E\|^\gamma\,d\mu(E)<\infty.
$$
\end{definition}

\begin{definition}\label{defn:uniform_conorm_moment}
Given a collection $\mc K$ of probability measures on $M_{d\times d}(\R)$,
has finite \emph{uniform translated $\gamma$-conorm moments} if for all $M$,
$$
\sup_{\mu\in\mc K}\sup_{\|A\|\le M}\int \|(A+E)^{-1}\|^\gamma\,d\mu(E)<\infty
$$    
\end{definition}

The following lemma establishes that if $\mc K$ is a collection of probability 
measures on $\R^{d\times d}$,
satisfying norm and conorm moment conditions, then the normalizations of $A+E$ where $E$ is distributed
as $\mu$ satisfy moment conditions, possibly for some smaller exponent.
\begin{lem}\label{lem:specialization}
Suppose that $\mc K$ is a collection of probability measures on $M_{d\times d}(\R)$ with 
finite uniform $\gamma$-norm moment and finite uniform translated 
$\gamma$-conorm moments for some $\gamma>0$. 
Then there exists $\gamma'>0$ such that for all $M$,
\begin{align*}
&\sup_{\|A\|\le M}\sup_{\mu\in\mc K}\int \big\|\widehat{A+E}\big\|^{\gamma'}\,d\mu(E)
<\infty\text{; and}\\
&\sup_{\|A\|\le M}\sup_{\mu\in\mc K}\int \big\|(\widehat{A+E})^{-1}\big\|^{\gamma'}\,d\mu(E)<\infty.\\
\end{align*}
In particular, $\bigcup_{\|A\|\le M}(G_A)_*(\mc K)$ is tight, where $G_A(E)\coloneqq \widehat{A+E}$.
\end{lem}

\begin{proof}
    Let $\gamma'=\gamma/(2-\frac 2d)$. For any matrix $B\in\GL(d,\R)$, we have the inequalities
    \begin{equation*}\label{eq:detbounds}
    \frac{\|B\|}{\|B^{-1}\|^{d-1}}=s_1(B)s_d(B)^{d-1}\le |\det B|
    \le s_1(B)^{d-1}s_d(B)=\frac{\|B\|^{d-1}}{\|B^{-1}\|}.
    \end{equation*}
    Since $\hat B=B/|\det B|^{1/d}$, substituting the bounds above, we deduce
    \begin{equation}\label{eq:hatbounds}
    \begin{split}
    \|\hat B\|&\le \|B\|^{(d-1)/d}\|B^{-1}\|^{(d-1)/d};\text{ and}\\
    \|\hat B^{-1}\|&\le \|B\|^{(d-1)/d}\|B^{-1}\|^{(d-1)/d}.
    \end{split}
    \end{equation}
    We now compute
    \begin{align*}
        &\sup_{\|A\|\le M}\sup_{\mu\in \mc K}\int \|\widehat{A+E}\|^{\gamma'}\,d\mu(E)\\
        &\le \sup_{\|A\|\le M}\sup_{\mu\in\mc K}
        \int \|A+E\|^{\gamma'(d-1)/d}\|(A+E)^{-1}\|^{\gamma'(d-1)/d}\,d\mu(E)\\
        &=\sup_{\|A\|\le M}\sup_{\mu\in\mc K}\int \|A+E\|^{\gamma/2}\|(A+E)^{-1}\|^{\gamma/2}\,d\mu(E)\\
        &\le\sup_{\|A\|\le M}\sup_{\mu\in\mc K}\left[\left(\int \|A+E\|^\gamma\,d\mu(E)\right)^{1/2}
        \left(\int \|(A+E)^{-1}\|^\gamma\,d\mu(E)\right)^{1/2}\right]\\
        &\le \left[\sup_{\mu\in\mc K}\left(\int (M+\|E\|)^\gamma\,d\mu(E)\right)^{1/2}\right]
        \left[\sup_{\|A\|\le M}\sup_{\mu\in\mc K}
        \left(\int \|(A+E)^{-1}\|^\gamma\,d\mu(E)\right)^{1/2}\right]\\  
        &\le \left[\sup_{\mu\in\mc K}
        \left(\int 2^\gamma(M^\gamma+\|E\|^\gamma)\,d\mu(E)\right)^{1/2}\right]
        \left[\sup_{\|A\|\le M}\sup_{\mu\in\mc K}
        \left(\int \|(A+E)^{-1}\|^\gamma\,d\mu(E)\right)^{1/2}\right].  
    \end{align*}
    By assumption, both terms are finite, giving the required bound. 
    Similarly, since by \eqref{eq:hatbounds} we have the same upper bounds for $\|\hat B\|$ and 
    $\|\hat B^{-1}\|$, we obtain the same estimate for $\sup_{\|A\|\le M}\sup_{\mu\in\mc K}\int\|
    (\widehat{A+E})^{-1}\|^{\gamma'}\,d\mu(E)$.
  
   To show that this implies that $\bigcup_{\|A\|\le M}(G_A)_*(\mc K)$ is tight, 
   let $L$ be defined by 
   $$
   L=\max\left(\sup_{\|A\|\le M}\sup_{\mu\in\mc K}\int \big\|\widehat{A+E}\big\|^{\gamma'}\,d\mu(E),
   \sup_{\|A\|\le M}\sup_{\mu\in\mc K}\int \big\|(\widehat{A+E})^{-1}\big\|^{\gamma'}\,d\mu(E)\right).
   $$
   In particular, for any $A$ with $\|A\|\le M$ and any $\mu\in\mc K$, we have the estimates
   $\int\|\widehat{A+E}\|^{\gamma'}\,d\mu(E)\le L$ and
   $\int\|(\widehat{A+E})^{-1}\|^{\gamma'}\,d\mu(E)\le L$. If $\epsilon>0$ is given, then
   the first estimate gives 
   $\mu(\{E\colon \|\widehat{A+E}\|>(2L/\epsilon)^{1/{\gamma'}}\})<\frac\epsilon2$.
   That is, 
   $$
   (G_A)_*\mu(\{B\colon \|B\|>(2L/\epsilon)^{1/\gamma'}\})<\tfrac\epsilon 2.
   $$ 
   Similarly, the second estimate gives 
   $$
   (G_A)_*\mu(\{B\colon \|B^{-1}\|>(2L/\epsilon)^{1/\gamma'}\})<\tfrac\epsilon 2.
   $$
   This establishes that for all $A$ with $\|A\|\le M$ and for all $\mu\in\mc K$,
   $$
   (G_A)_*\mu(\{B\colon (2L/\epsilon)^{-1/\gamma'}\le s_d(B)\le s_1(B)\le (2L/\epsilon)^{1/\gamma'}\})
   >1-\epsilon.
   $$
   Since $\{B\in \SL(d,\R)\colon a\le s_d(B)\le s_1(B)\le b\}$ is a compact set for all $a,b$,
   this establishes the required tightness of $\bigcup_{\|A\|\le M}(G_A)_*(\mc K)$.    
\end{proof}

\begin{lem}\label{lem:singleton}
    Let $\mc K$ be a set of absolutely continuous probability measures on $\GL(d,\R)$
    with uniformly bounded density (i.e. $\sup_{\mu\in\mc K}\sup_{A\in\GL(d,\R)}\frac{d\mu}{d\vol}(A)
    <\infty$)
    and
    uniformly bounded $\gamma$-norm moment. Then $\mc K$ satisfies the assumptions of
    Lemma \ref{lem:specialization}.
\end{lem}

\begin{proof}
    It suffices to check that $\mc K$ has a finite uniform translated $\gamma'$-conorm moment 
    for some $\gamma'>0$. In the course of the proof, the letter $C$ will denote various constants 
    that do not depend on $A$, and may vary from line to line. We recall that $s_d(A)=\|A^{-1}\|^{-1}$.
    We make use of the fact that there is a constant $C$
    such that for any matrix $A$,
    \begin{equation}\label{eq:Leb_bound}
    \text{Leb}(\{B\colon \|B\|\le 2^n,\ s_d(A+B)<\alpha\})\le 
    C2^{(d^2-1)n}\alpha\text{ for all $\alpha$}.
    \end{equation}
    To see this, notice that if $s_d(A+B)<\alpha$, then there is a unit vector $v$ such that 
    $\|(A+B)v\|<\alpha$. If $v$ is such a vector and 
    $|v_i|\ge \frac 1{\sqrt d}$, then the $i$th column of $A+B$
    is no further than $\alpha\sqrt d$ from the span of the other columns, so that 
    $\{B\colon \|B\|\le 2^n,\ s_d(A+B)<\alpha\}$ is contained in the $d$-fold union
    of sets of volume $O((2^n)^{d^2-1}\alpha)$.
   
    Since $\mc K$ has a finite $\gamma$-norm moment, we observe from Markov's inequality, there
    exists a $C$ such that for all $\mu\in \mc K$,
\begin{equation}
    \mu(\{B\colon 2^n\le \|B\|<2^{n+1}\})\le C/2^{-\gamma n},
\end{equation}
(where $C$ is the $\gamma$-norm).

Let $s\in\N$ and $\mu\in\mc K$ be fixed. Let $\delta>0$ be a constant to be chosen 
later (independently of $s$ and $\mu$). 
If $\|A\|\le 2^s$, we now estimate
\begin{align*}
    &\int \|(A+E)^{-1}\|^\delta\,d\mu(E)\le\circled{1}+\circled{2}+\circled{3}\text{ where}\\
    \circled{1}&=C\int_{\|A+E\|\le 2^{s+1}}\|(A+E)^{-1}\|^{\delta}\,d\text{Leb}(E);\\
    \circled{2}&=\mu\{E\colon s_d(A+E)\ge 1\};\text{ and}\\
    \circled{3}&=\sum_{n=s+1}^\infty\sum_{m=0}^\infty 2^{\delta (m+1)}
    \mu(\{E\colon 2^n\le \|A+E\|<2^{n+1},
    2^{-(m+1)}\le s_d(A+E)<2^{-m}\}),
\end{align*}
the constant $C$ in the first line being the uniform bound on the density. 
Clearly, $\circled{2}\le 1$. For $\circled{1}$, we estimate
$$
\circled{1}\le C\int_{\|E\|\le 2^{s+2}}\|E^{-1}\|^\delta\,d\text{Leb}(E).
$$
For $\delta<1$, this can be seen to be finite using \eqref{eq:Leb_bound}.
For $\circled{3}$, since $\|A\|<2^{n-1}$ for the $n$'s we consider, we observe
\begin{align*}
    &\mu(\{E\colon 2^n\le \|A+E\|<2^{n+1}, 2^{-(m+1)}\le s_d(A+E)<2^{-m}\})\\
    &\le \min\Big(\mu(\{E\colon 2^n\le \|A+E\|<2^{n+1}\}),C\,\text{Leb}(\{E\colon\|A+E\|<2^{n+1},
    s_d(A+E)<2^{-m}\})\Big)\\
    &\le C\min\Big(\mu(\{E\colon 2^{n-1}\le \|E\|\}),
    \text{Leb}(\{E\colon \|E\|<2^{n+2}, s_d(E)<2^{-m}\})\Big)\\
    &\le C\min(2^{-n\gamma},2^{n(d^2-1)-m}),
\end{align*}
where the same constant $C$ works for all $\mu\in\mc K$.
Hence we need to estimate
$$
\sum_{n=s}^\infty\sum_{m=0}^\infty 2^{\delta m}\min(2^{-n\gamma},2^{n(d^2-1)-m}).
$$
Provided $\delta<1$, for fixed $n$, the summands increase geometrically in $m$ with 
common ratio $2^\delta$ up to $m=n(d^2-1+\gamma)$ (at which point the two terms in the minimum agree)
and then decrease geometrically in $m$ with common ratio 2. 
Accordingly, the sum is bounded above by a constant multiple of 
$$
\sum_{n=s}^\infty 2^{\delta n(d^2-1+\gamma)}2^{-n\gamma}.
$$
Letting $\delta$ be $\frac \gamma 2/(d^2-1+\gamma)$, for example, ensures that the sum is finite. 
Hence we have shown that $\mc K$ has finite translated $\delta$-conorm moment. 
\end{proof}

\subsection{Proof of the Main Theorem. }
We first state our main theorem in a more general form and then state a more 
precise quantitative version of it. 
\begin{thm}\label{thm:main_thm_abstract}
For any dimension $d$ and any $M>0$, the following holds. Suppose that 
$(\mu_n)_{n\in \N}$ is a sequence of distributions on $M_{d\times d}(\R)$ satisfying 
\begin{enumerate}
    \item The collection $\mc K=\{\mu_n\}$ has a uniform $\gamma$-norm moment and uniform translated
    $\gamma$-conorm moment for some $\gamma>0$. 
   \item \label{it:eta_entropy}
    For any matrix $A$ with $\|A\|\le M$, $(G_A)_*\mu_n$ has a uniform fiberwise 
    push-forward entropy (Def.~\ref{defn:uniform_fiberwise_pushforward_entropy}) 
    of at least $\eta$ on the miniflags
    of core dimension $k$ for $0\le k\le d-2$, where $G_A(E)=A+E$.
\end{enumerate}
Then for any sequence of $d\times d$ matrices $(A_n)_{n\in \N}$ such that
$\|A_n\|\le M$ for each $n$, define a random product by
\[
B^{n}\coloneqq (A_n+E_n)(A_{n-1}+E_{n-1})\cdots(A_1+E_1),
\]
where the $(E_i)$ are independent and the distribution of $E_i$ is given by $\mu_i$. 
Then we have
\begin{equation}\label{eqn:expectation_grows_general}
\E\left[\log \sval_k(B^n)-\log \sval_{k+1}(B^n)\right]\ge n\eta-C_d,
\end{equation}
where $C_d$ is as in the statement of Proposition \ref{prop:gapest}; and almost surely
\begin{equation}\label{eqn:almost_sure_est}
\liminf_{n\to \infty} \tfrac 1n\big(\log \sval_k(B^n)-\log \sval_{k+1}(B^n)\big)\ge \eta.
\end{equation}
\end{thm}

\begin{proof}
It suffices to show the estimate \eqref{eqn:expectation_grows_general}; equation  
\eqref{eqn:almost_sure_est} then follows by Proposition 
\ref{prop:growth_in_expectations_implies_almost_sure}, 
whose hypotheses are satisfied due to Lemma \ref{lem:specialization} and the assumption (1) above.

Fix some $0\le k-1\le d-2$. By assumption \eqref{it:eta_entropy},
for all $n$, the distribution of $\nu_n=(G_{A_n})_*\mu_n$
has fiber-averaged push-forward entropy at least
$\eta$ on the miniflags of core dimension $k-1$. Then by Proposition 
\ref{prop:additivity_of_pushforward_entropy}, it follows 
that it follows that the distribution of $B^n$ satisfies that 
\[
\EPIA(\mc F_0,\nu_n*\cdots *\nu_1)\ge n\eta. 
\]
From Proposition \ref{prop:miniflag_distortion_implies_singular_value_gaps}, the estimate 
\eqref{eqn:expectation_grows_general} is now immediate.
\end{proof}

We can now deduce the main consequence of this, which is the theorem stated in the introduction. 

\begin{proof}[Proof of Theorem~\ref{thm:main_thm_simple}.]
We apply Theorem \ref{thm:main_thm_abstract} with $(\mu_n)$ being the constant sequence
$(\mu)$. Recall that $\mu$ is absolutely continuous with continuous density $\phi$
satisfying $\phi(E)\le C/\|E\|^{d^2+1}$. Let $K=\max\phi$.
For $0<\gamma<1$, we estimate the $\gamma$-norm moment of $\mu$ by
\begin{align*}
    \int\|E\|^\gamma\,d\mu(E)&=\int \phi(E)\|E\|^\gamma\,d\text{Leb}(E)\\
    &\le A\int_0^\infty \min(K,C{r^{-(d^2+1)}})r^\gamma r^{d^2-1}\,dr\\
    &\le AK\int_0^1 r^{d^2+\gamma-1}\,dr+AC\int_1^\infty r^{\gamma-2}\,dr<\infty.
\end{align*}
Hence Lemmas \ref{lem:singleton} and \ref{lem:specialization} imply
hypothesis (1) of the theorem is satisfied for the 
family $\mc K=\{\mu\}$.
Hypothesis \eqref{it:eta_entropy} of the theorem (with $\eta=\epsilon^2C_\phi/M^2$) 
is satisfied by
Proposition \ref{prop:fiberwise_entropy_invertible_case}.
\end{proof}

\begin{rem}
It is also possible to take the noise $\mu_{\epsilon}$ to be non-stationary but drawn from 
a family of distributions $\phi$ that is precompact.
\end{rem}

\appendix
\section{Upper bounds}\label{app:upper_bds}

For $L>0$, let $\mc M_L$ be the set of probability measures on $M_{d\times d}(\R)$ with the properties:
\begin{itemize}
    \item $\mu\{E\colon \|E\|\le L\}=1$;
    \item $\int E_{ij}\,d\mu(E)=0$ for all $i,j$.
    \item $\int E_{ij}E_{kl}\,d\mu(E)=\delta_{ik}\delta_{jl}$ for all $1\le i,j,k,l\le d$.
\end{itemize}

\begin{lem}\label{lem:conj}
Let $\SO(d)$ act on $M_{d\times d}(\R)$ by conjugation, denoted by $O(E)=OEO^{-1}$.
Then $O_*(\mc M_L)=\mc M_L$ for each $O\in \SO(d)$. 
\end{lem}

\begin{proof}
    Let $m\in\mc M_L$. For $O\in \SO(d)$, since $\|E\|=\|OEO^{-1}\|$, $O_*\mu$ satisfies the first condition. 
    The measure $O_*\mu$ satisfies the second condition follows by linearity. 
    For the third condition, we have
    \begin{align*}
        \int (OEO^{-1})_{ab}(OEO^{-1})_{cd}\,d\mu(E)&=\sum_{i,j,k,l}O_{ai}O^{-1}_{jb}O_{ck}O^{-1}_{ld}\int E_{ij}E_{kl}\,d\mu(E)\\
        &=\sum_{i,j,k,l}O_{ai}O^{-1}_{jb}O_{ck}O^{-1}_{ld}\delta_{ik}\delta_{jl}\\
        &=\sum_{i,j}O_{ai}O_{ci}O^{-1}_{jb}O^{-1}_{jd}\\
        &=(OO^T)_{ac}(OO^T)_{db}\\
        &=\delta_{ac}\delta_{db},
    \end{align*}
    as required. Hence $O_*\mc M_L\subseteq M_L$. Since $O^{-1}_*\mc M_L\subseteq \mc M_L$ also, it follows that
    $O_*\mc M_L=\mc M_L$ as required.
\end{proof}

\begin{lem}\label{lem:second_order_taylor_expansion}
Let $A>0$ be fixed. Then there exists $\epsilon_0>0$ such that if $Y$, $X_1$, $X_2$ 
are random variables such that $|X_1|\le A$, $|X_2|\le A$ and 
$|Y-(1+\epsilon X_1+\epsilon^2X_2)|\le A\epsilon^3$
almost surely, then for all $0<\epsilon<\epsilon_0$,
\[
\E\,\log Y=\epsilon \mathbb{E}X_1 + \epsilon^2 \big(\E X_2-\tfrac12\E X_1^2\big)+O(\epsilon^3).
\]
\end{lem}

\begin{proof}
We will use the Taylor expansion with remainder. Let $f(x)=\log(1+x)$
and write $Y=1+Z$ with $Z=\epsilon X_1+\epsilon^2X_2+R$, where $|R|\le A\epsilon^3$ a.s. Then
\[
\log Y=Z-\tfrac 12Z^2+\tfrac 16f'''(z)Z^3
\]
where $z$ is between $0$ and $Z$. 
By compactness of the support of the random variables we are considering, 
and truncating the expansions of $Z$ and $Z^2$ at the $\epsilon^2$ terms,
the error term is of size $O_A(\epsilon^3)$, and the stated estimate follows. 
\end{proof}

\begin{lem}\label{lem:Taylor}
    Let $1\le k\le d$ and let $M_{d\times d}$ act on $\bigwedge^k\R^d$ by $A(v_1\wedge\ldots\wedge v_k)
    =Av_1\wedge \ldots \wedge Av_k$. Then for all $\mu\in\mc M_L$, 
    $$
    \int\log\left\|(I+\epsilon E)(e_1\wedge\ldots\wedge e_k)\right\|\,d\mu(E)=\frac{k(d-k-1)\epsilon^2}{2}+O(\epsilon^3)
    $$ 
\end{lem}

\begin{proof}
Let $\mu\in\mc M_L$ and let $\V$ be the random variable 
$(I+\epsilon E)(e_1\wedge\ldots\wedge e_k)$, where $E$
is distributed as $\mu$. That is,
$$
\V=((I+\epsilon E)e_1)\wedge\ldots\wedge ((I+\epsilon E)e_k)
$$
Expanding the right side, we have
\begin{align*}
\V=&(e_1+\epsilon E_{11}e_1+\epsilon E_{21}e_2+\ldots+\epsilon E_{d1}e_d)\\
&\wedge(e_2+\epsilon E_{12}e_1+\epsilon E_{22}e_2+\ldots+\epsilon E_{d2}e_d)\\
&\wedge\quad\vdots\\
&\wedge (e_k+\epsilon E_{1k}e_1+\epsilon E_{2k}e_2+\ldots+\epsilon E_{dk}e_d).
\end{align*}
We need to compute $\E\log\|\V\|=\frac 12\E\log\|\V\|^2$. Recall that, by definition, 
if we write $\V$ as $\sum_{i_1<\ldots<i_k} a_{i_1\cdots i_k} e_{i_1}\wedge \cdots \wedge e_{i_k}$,
then $\|\V\|^2=\sum_{i_1<\ldots<i_k} a_{i_1\cdots i_k}^2$. Thus, in view of Lemma 
\ref{lem:second_order_taylor_expansion}, to compute $\E{\log\|\V\|^2}$ it suffices 
to study the sum $\sum_{i_1<\ldots<i_k} a_{i_1\cdots i_k}^2$ up to an
error of order $O(\epsilon^3)$.

We now examine the cross terms arising from the above expression for $\V$ with this in mind. 
Any term in $a_{i_1\ldots i_k}$ that is a multiple of $\epsilon^3$ or a higher power
may be discarded; similarly if $a_{i_1\ldots i_k}$ has no constant term, then any term
of order $\epsilon^2$ may be discarded since we sum the squares of the coefficients. 
Thus we see that there are two kinds of terms in $\V$ that give a contribution 
of order $\epsilon^2$ or greater to $\|\V\|^2$:
\[
e_1\wedge \cdots \wedge e_k\text{ and } e_1\wedge\ldots\wedge e_{i-1}\wedge 
e_\ell\wedge e_{i+1}\wedge\ldots\wedge e_k \text{ where }\ell>k.
\]
Let us examine these two cases:
\begin{enumerate}
    \item The coefficient of $e_1\wedge \cdots \wedge e_k$ (up to order $\epsilon^2$) is 
    \[
1+\epsilon \sum_{i=1}^kE_{ii}+
\epsilon^2\left(\sum_{i<j}(E_{ii}E_{jj}+(-1)^{j-i}
E_{ji}E_{ij})\right).
    \]
    \item
    Next, each term of the form 
    $e_1\wedge\ldots\wedge e_{i-1}\wedge e_\ell\wedge e_{i+1}\wedge\ldots\wedge e_k$
    up to order $\epsilon$ arises in a unique way: from each term of the wedge product
    expression for $\V$, we select the principal term except that in the $i$th term we select
    the $\epsilon E_{\ell i}e_\ell$ term.
\end{enumerate}

Squaring the coefficients and summing, we observe that $\|\V\|^2=
1+\epsilon X_1+\epsilon^2 X_2+O(\epsilon^3)$, where
\begin{align*}
    X_1&=2\sum_{i=1}^k E_{ii}; \text{\quad and}\\
    X_2&=\left(\sum_{i=1}^k E_{ii}\right)^2+2\sum_{i<j}\Big(E_{ii}E_{jj}+
    (-1)^{j-i}E_{ji}E_{ij}\Big)+\sum_{i=1}^k\sum_{\ell=k+1}^dE_{\ell i}^2.
\end{align*}
We observe from the expectation and covariance properties that $\E X_1=0$, $\E X_1^2=4k$
and $\E X_2=k+0+k(d-k)=k(d+1-k)$.  Applying Lemma \ref{lem:second_order_taylor_expansion}, we obtain
$$
\E\log\|\V\|^2=k(d-1-k)\epsilon^2+O(\epsilon^3).
$$

That gives
$$
\E\log\|\V\|=\tfrac 12\E\log\|\V\|^2=\frac{k(d-k-1)}2\epsilon^2+O(\epsilon^3),
$$
as required.
\end{proof}

\begin{cor}\label{cor:rank_1}
    For any $\mu\in\mc M_L$ and
    any rank 1 element, $v_1\wedge\ldots\wedge v_k$, of $\bigwedge^k\R^d$ of norm 1,
    $$
    \E\log\|(I+\epsilon E)(v_1\wedge\ldots\wedge v_k)\|=\frac{k(d-k-1)\epsilon^2}{2}+O(\epsilon^3).
    $$
\end{cor}

\begin{proof}
    By transitivity of the action of $\SO(d)$ on orthogonal $k$-frames, there
    exists $O\in \SO(d)$ such that $O(e_1\wedge\ldots\wedge e_k)=v_1\wedge\ldots\wedge v_k$. 
    Thus 
    \begin{align*}
    \|(I+\epsilon E)(v_1\wedge \ldots\wedge v_k)\|
    &=\|((I+\epsilon E)O)(e_1\wedge \ldots\wedge e_k)\|\\
    &=\|(O^{-1}(I+\epsilon E)O)(e_1\wedge\ldots\wedge e_k)\|\\
    &=\|(I+\epsilon O^{-1}EO)(e_1\wedge\ldots\wedge e_k)\|.
\end{align*}
    Since $(O^{-1})_*\mu\in\mc M_L$ by Lemma \ref{lem:conj}, the result follows from 
    Lemma \ref{lem:Taylor}.
\end{proof}

\begin{thm}
    Let $(\mu_n)_{n\in\N}$ be an arbitrary sequence of elements of $\mc M_L$ and
    let $(E_n)_{n\in\N}$ be an independent sequence of matrix random variables, 
    with $E_n$ having distribution $\mu_n$. 
    For $\epsilon>0$, let $M_n=(I+\epsilon E_n)\cdots (I+\epsilon E_1)$.
    Then a.s. for each $1\le k\le d$,
    \begin{equation}\label{eqn:growth_rate_k_d-2k}
    \limsup_{n\to\infty} \left|\frac 1n\log\sigma_k(M_n)-\frac {(d-2k)\epsilon^2}{2}\right|=O(\epsilon^3).
    \end{equation}
    In particular, for each $1\le k<d$,
    $$
    \limsup_{n\to\infty} \left|\frac 1n\log\frac{\sigma_k(M_n)}{\sigma_{k+1}(M_n)}-\epsilon^2\right|=O(\epsilon^3).
    $$
\end{thm}

\begin{proof}
To begin, note that here we are studying a stationary matrix product given by a measure 
$\mu$ on $\GL(d,\R)$. 
Thus, in this case Lyapunov exponents exist.
As the Lyapunov exponents are the exponential grow rate of these singular values, 
equation \eqref{eqn:growth_rate_k_d-2k} will follow once we estimate the Lyapunov 
exponents as $n^{-1}\log \sigma_k(M_n)$ converges to $\lambda_k$. Hence it suffices to estimate $\lambda_k$.

Consider a stationary measure $\nu$ for the dynamics on $\Gr(k,d)$. 
Let $\Phi_k(A,V)\colon \GL(d,\R)\times \Gr(k,d)\to \R$ be given by 
\[
\Phi_k(A,v)=\log \frac{\|A(v_1\wedge \cdots \wedge v_k)\|}{\|v_1\wedge \cdots \wedge v_k\|},
\]
where $v_1\wedge \cdots \wedge v_k$ is any wedge representing the subspace $V\in \Gr(k,d)$. 
Then Furstenberg's formula \cite{goldsheid1989lyapunov} gives that there is some 
stationary measure $\nu$ such that:
\[
\lambda_1+\cdots+\lambda_k=\int_{\Gr(k,d)}\int_{\GL(d,\R)} \Phi(A,V)\,d\mu(A)\,d\nu(V).
\]
But note now that by the Corollary \ref{cor:rank_1}, we have already evaluated the 
inner integral for all rank-$1$ tensors, so the result follows by taking the 
difference between $\lambda_1+\cdots+\lambda_{k}$ and $\lambda_1+\cdots+\lambda_{k-1}$.
\end{proof}

\section{Fiber measure density calculations}\label{app:coarea}
We include here a brief justification of \eqref{eq:xi3fibredensity} and \eqref{eq:xi4fibredensity}.
First for the $\Xi_3$ fiber density, 
there is a fixed orthonormal basis $v_1,\ldots,v_d$. The partition $\Xi_2$ 
then partitions $M_{d\times d}(\R)$ according to the values of $Av_1,\ldots Av_{k-1},
Av_{k+2},\ldots,Av_d$. It also specifies 
$\Pi_{W_{k-1}}Av_k$ and $\Pi_{W_{k-1}}Av_{k+1}$. The fiber measure on leaves is just Lebesgue
measure. Since Lebesgue measure on the matrices is
invariant under pre- and post- multiplication by orthogonal matrices, we may assume without 
loss of generality that $v_1,\ldots,v_d$ form an orthonormal coordinate basis and that 
$\lin(w_1,\ldots,w_{k-1})$ is the linear span of the first $k-1$ orthonormal coordinate
vectors. That is, each $\xi_2$ fiber specifies all entries of $A$ except 
$A_{ij}$ for $i=k,k+1$ and $j\ge k$ which are completely unspecified. Fixing a $\xi_2$-fiber,
we restrict attention to the free $2(d-k+1)$ coordinates. Two matrices belong to the same 
$\xi_3$-fiber if they share the same linear span of these two $(d-k+1)$-vectors. 
We therefore study the map sending a pair of $(d-k+1)$-vectors to their span, given by
$F\colon (\R^{d-k+1})^2\to \Gr_2(\R^{d-k+1})$ given by
$F(u,v)=\lin(u,v)$. 

We then use the co-area formula \cite{Chavel}:
$$
\int f(x)\,dm=\int_{\Gr_2}\int_{F^{-1}(W)} \frac{f(x)}{J_\text{norm}F(x)}\,dm_W(x)\,dm_{\Gr_2}(W),
$$
where $m$ is Lebesgue measure on $(\R^{d-k+1})^2$; $m_{\Gr_2}$ is Haar measure on $\Gr_2(\R^{d-k+1})$;
$m_W$ is Lebesgue measure on $W$; and $J_\text{norm}F$ is the normal Jacobian, 
i.e.\ the determinant of the Jacobian of $F$
restricted to the orthocomplement of the kernel of the derivative. From the formula,
the fiber measure density is proportional to $1/J_\text{norm}F(x)$. 

Applying an orthogonal change of coordinates, it suffices to compute $J_\text{norm}F(u_0,v_0)$ where 
$u_0,v_0\in\lin(e_1,e_2)$. We use a standard chart for the Grassmannian in a neighborhood of 
$\lin(e_1,e_2)$: a 2-plane $V$ in the neighborhood is parameterized 
by a pair of vectors $w,x$ in $\lin(e_3,\ldots,e_{d-k+1})$
where $e_1+w$ and $e_2+x$ are the unique elements of $V\cap \{x\colon x_1=1,x_2=0\}$ and 
$V\cap \{x\colon x_1=0,x_2=1\}$ respectively. In other words, we represent this 
subspace $V$ as a graph of a map $L_V\colon \lin(e_1,e_2)\to \lin(e_1,e_2)^{\perp}$, 
and these vectors $w$ and $x$ are equal to $L_V(e_1)$ and $L_V(e_2)$, respectively. 

Hence, in these charts we can express $F$ as a map from 
$\R^{2d-2k+2}$ to $\R^{2d-2k-2}$ mapping a pair of vectors $(u,v)$ to the pair $(w,x)$ describing the
subspace $\lin(u,v)$. 

Since $u_0,v_0$ are in $\lin(e_1,e_2)$, $F(u_0,v_0)=(0,0)$. To compute the normal Jacobian, 
notice that if $u_0$ or $v_0$ are 
perturbed in the $\lin(e_1,e_2)$ plane, the image under $F$ remains unchanged,  
hence they both must lie in $\ker DF$. In fact, this is the entire kernel, 
so we now compute the derivative on its orthogonal complement. We therefore compute the derivative
when $u_0$ and $v_0$ are perturbed in the $\lin(e_3,\ldots,e_{d-k+1})$ subspace. Let $\alpha$, $\beta$,
$\gamma$ and $\delta$ be such that $\alpha u_0+\beta v_0=e_1$, $\gamma u_0+\delta v_0=e_2$. 
the map $F$ restricted to $(u_0+\lin(e_3,\ldots,e_{d-k+1}))\times
(v_0+\lin(e_3,\ldots,e_{d-k+1}))$, with respect to the chart above is linear,
given by $F(u_0+y,v_0+z)=(\alpha y+\beta z,\gamma y+\delta z)$. The Jacobian matrix is therefore
$$
\begin{pmatrix}
\alpha&&&\gamma&&\\
&\ddots&&&\ddots&\\
&&\alpha&&&\gamma\\
\beta&&&\delta&&\\
&\ddots&&&\ddots&\\
&&\beta&&&\delta\\
\end{pmatrix};
$$
and the Jacobian determinant is $|\alpha\delta-\beta\gamma|^{d-k-1}$. Since 
$(\alpha u_0+\beta v_0)\wedge (\gamma u_0+\delta v_0)=e_1\wedge e_2$, 
we see that $\|u_0\wedge v_0\|=1/|\alpha\delta-\beta\gamma|$,
so that $1/J_\text{norm}F(u_0,v_0)=\|u_0\wedge v_0\|^{d-k-1}$. 

To compute the $\Xi_4$ fiber density, we again assume that $v_1,\ldots,v_d$
is the coordinate basis and $w_1,\ldots,w_{k-1}\in\lin(e_1,\ldots,e_{k-1})$.
We further assume without loss of generality that $U=\lin(e_k,e_{k+1})$, $y=e_k$ and $y^\perp=e_{k+1}$. 
In terms of the matrix, all the entries are specified except $A_{ij}$ for $i,j\in\{k,k+1\}$, which
is completely unspecified. We write $\bar A$ for the $2\times 2$ submatrix with these entries. 
The partition element $\xi_4(\theta)$ consists of those $A$ such that
$\Pi_UA(\cos\theta v_k+\sin\theta v_{k+1})$ is parallel to $y$. That is, those $A$ such that
$\langle y^\perp,A(\cos\theta v_k+\sin\theta v_{k+1})\rangle=0$. 
Since the only elements that are varying are those in the submatrix, the partition is equivalent
to a measurable partition $\bar\Xi_4$ of the submatrices. 
Writing $\bar A=\left(\begin{smallmatrix}a&b\\c&d\end{smallmatrix}\right)$, the partition elements
are the level sets of $\bar F(\bar A)=-\tan^{-1}(\frac cd)$. As before, the density multiplier is the 
normal Jacobian. 
We verify that 
$$
\bar F^{-1}(\theta)=\left\{\begin{pmatrix}a&b\\-c\sin\theta&c\cos\theta\end{pmatrix}\colon a,b,c\in\R\right\}.
$$
The normal Jacobian to $F$ at this point is the derivative of $F$ in the 
$\left(\begin{smallmatrix}0&0\\\cos\theta&\sin\theta\end{smallmatrix}\right)$ direction. 
Hence, one can see that the reciprocal normal Jacobian is $|c|$, which can be written
in a coordinate-free way as $|\langle y^\perp,A(-\sin\theta v_k+\cos\theta v_{k+1})\rangle$.
This allows one to deduce \eqref{eq:xi4fibredensity} as claimed.
\printbibliography

@misc{atnip2023universal,
      title={Universal Gap Growth for Lyapunov Exponents of Perturbed Matrix Products}, 
      author={Jason Atnip and Gary Froyland and Cecilia González-Tokman and Anthony Quas},
      year={2023},
      eprint={2312.03181},
      archivePrefix={arXiv},
      primaryClass={math.DS},
      url={https://arxiv.org/abs/2312.03181}, 
}

@book {benoist2016random,
    AUTHOR = {Benoist, Yves and Quint, Jean-{Fran\c cois}},
     TITLE = {Random walks on reductive groups},
    SERIES = {Ergebnisse der Mathematik und ihrer Grenzgebiete. 3. Folge. A
              Series of Modern Surveys in Mathematics [Results in
              Mathematics and Related Areas. 3rd Series. A Series of Modern
              Surveys in Mathematics]},
    VOLUME = {62},
 PUBLISHER = {Springer, Cham},
      YEAR = {2016},
     PAGES = {xi+323},
   MRCLASS = {60B15 (22E30 22E46 37C30 37H15 60B20 60F05 60G50)},
  MRNUMBER = {3560700},
MRREVIEWER = {Radhakrishnan\ Nair},
}

@article {goldsheid1989lyapunov,
    AUTHOR = {Gol'dshe{\u i}d, I. Ya. and Margulis, G. A.},
     TITLE = {Lyapunov exponents of a product of random matrices},
   JOURNAL = {Uspekhi Mat. Nauk},
  FJOURNAL = {Akademiya Nauk SSSR i Moskovskoe Matematicheskoe Obshchestvo.
              Uspekhi Matematicheskikh Nauk},
    VOLUME = {44},
      YEAR = {1989},
    NUMBER = {5(269)},
     PAGES = {13--60},
      ISSN = {0042-1316},
   MRCLASS = {60B15 (22E30 47E05 60H25 82B44)},
  MRNUMBER = {1040268},
MRREVIEWER = {Philippe\ Bougerol},
       DOI = {10.1070/RM1989v044n05ABEH002214},
       URL = {https://doi.org/10.1070/RM1989v044n05ABEH002214},
}

@article {blumenthal2017lyapunov,
    AUTHOR = {Blumenthal, Alex and Xue, Jinxin and Young, Lai-Sang},
     TITLE = {Lyapunov exponents for random perturbations of some
              area-preserving maps including the standard map},
   JOURNAL = {Ann. of Math. (2)},
  FJOURNAL = {Annals of Mathematics. Second Series},
    VOLUME = {185},
      YEAR = {2017},
    NUMBER = {1},
     PAGES = {285--310},
      ISSN = {0003-486X,1939-8980},
   MRCLASS = {37E30 (37H15)},
  MRNUMBER = {3583355},
MRREVIEWER = {M.\ L.\ Blank},
       DOI = {10.4007/annals.2017.185.1.5},
       URL = {https://doi.org/10.4007/annals.2017.185.1.5},
}

@book{Chavel,
   author={Chavel, Isaac},
   title={Riemannian geometry},
   series={Cambridge Studies in Advanced Mathematics},
   volume={98},
   edition={2},
   date={2006},
}

@article {cowieson2005SRB,
    AUTHOR = {Cowieson, William and Young, Lai-Sang},
     TITLE = {S{RB} measures as zero-noise limits},
   JOURNAL = {Ergodic Theory Dynam. Systems},
  FJOURNAL = {Ergodic Theory and Dynamical Systems},
    VOLUME = {25},
      YEAR = {2005},
    NUMBER = {4},
     PAGES = {1115--1138},
      ISSN = {0143-3857,1469-4417},
   MRCLASS = {37C40 (37A35 37D05 37H15)},
  MRNUMBER = {2158399},
MRREVIEWER = {Pei\ Dong\ Liu},
       DOI = {10.1017/S0143385704000604},
       URL = {https://doi.org/10.1017/S0143385704000604},
}

@article {ledrappier2023exact,
    AUTHOR = {Ledrappier, {Fran\c cois} and Lessa, Pablo},
     TITLE = {Exact dimension of {F}urstenberg measures},
   JOURNAL = {Geom. Funct. Anal.},
  FJOURNAL = {Geometric and Functional Analysis},
    VOLUME = {33},
      YEAR = {2023},
    NUMBER = {1},
     PAGES = {245--298},
      ISSN = {1016-443X,1420-8970},
   MRCLASS = {37C45 (28A80 37A99)},
  MRNUMBER = {4561150},
       DOI = {10.1007/s00039-023-00631-0},
       URL = {https://doi.org/10.1007/s00039-023-00631-0},
}

@article {bedrossian2022regularity,
    AUTHOR = {Bedrossian, Jacob and Blumenthal, Alex and Punshon-Smith, Sam},
     TITLE = {A regularity method for lower bounds on the {L}yapunov
              exponent for stochastic differential equations},
   JOURNAL = {Invent. Math.},
  FJOURNAL = {Inventiones Mathematicae},
    VOLUME = {227},
      YEAR = {2022},
    NUMBER = {2},
     PAGES = {429--516},
      ISSN = {0020-9910,1432-1297},
   MRCLASS = {60H15 (35B65 35H10 37D25 37H15 58J65)},
  MRNUMBER = {4372219},
MRREVIEWER = {Xiaobin\ Sun},
       DOI = {10.1007/s00222-021-01069-7},
       URL = {https://doi.org/10.1007/s00222-021-01069-7},
}

@article {baxendale2022lyapunov,
    AUTHOR = {Baxendale, Peter H. and Goukasian, Levon},
     TITLE = {Lyapunov exponents for small random perturbations of
              {H}amiltonian systems},
   JOURNAL = {Ann. Probab.},
  FJOURNAL = {The Annals of Probability},
    VOLUME = {30},
      YEAR = {2002},
    NUMBER = {1},
     PAGES = {101--134},
      ISSN = {0091-1798,2168-894X},
   MRCLASS = {60H10 (37H15 37J99)},
  MRNUMBER = {1894102},
MRREVIEWER = {Klaus\ R.\ Schenk-Hopp\'e},
       DOI = {10.1214/aop/1020107762},
       URL = {https://doi.org/10.1214/aop/1020107762},
}

@article {sowers2001tangent,
    AUTHOR = {Sowers, Richard B.},
     TITLE = {On the tangent flow of a stochastic differential equation with
              fast drift},
   JOURNAL = {Trans. Amer. Math. Soc.},
  FJOURNAL = {Transactions of the American Mathematical Society},
    VOLUME = {353},
      YEAR = {2001},
    NUMBER = {4},
     PAGES = {1321--1334},
      ISSN = {0002-9947,1088-6850},
   MRCLASS = {60H10 (58J65)},
  MRNUMBER = {1806739},
MRREVIEWER = {Volker\ Wihstutz},
       DOI = {10.1090/S0002-9947-00-02773-2},
       URL = {https://doi.org/10.1090/S0002-9947-00-02773-2},
}

@article {auslender1982asymptotic,
    AUTHOR = {Auslender, E. I. and Mil'shte{\u i}n, G. N.},
     TITLE = {Asymptotic expansions of the {L}iapunov index for linear
              stochastic systems with small noise},
JOURNAL = {Journal of Applied Mathematics and Mechanics},
VOLUME = {46},
YEAR = {1982},
NUMBER ={3},
     PAGES = {277--283},
   MRCLASS = {93E15},
  MRNUMBER = {709667},
MRREVIEWER = {G.\ S.\ Ladde},
}

@article {bedrossian2024chaos,
    AUTHOR = {Bedrossian, Jacob and Punshon-Smith, Sam},
     TITLE = {Chaos in stochastic 2d {G}alerkin-{N}avier-{S}tokes},
   JOURNAL = {Comm. Math. Phys.},
  FJOURNAL = {Communications in Mathematical Physics},
    VOLUME = {405},
      YEAR = {2024},
    NUMBER = {4},
     PAGES = {Paper No. 107, 42},
      ISSN = {0010-3616,1432-0916},
   MRCLASS = {60H15 (17B05 32W50 35H10 35Q30 76D05 76M35)},
  MRNUMBER = {4733339},
MRREVIEWER = {Xiaobin\ Sun},
       DOI = {10.1007/s00220-024-04949-0},
       URL = {https://doi.org/10.1007/s00220-024-04949-0},
}

@misc{bedrossian2024quantitative,
      title={A quantitative dichotomy for Lyapunov exponents of non-dissipative SDEs with an application to electrodynamics}, 
      author={Jacob Bedrossian and Chi-Hao Wu},
      year={2024},
      eprint={2406.00220},
      archivePrefix={arXiv},
      primaryClass={math.PR},
      url={https://arxiv.org/abs/2406.00220}, 
}

@article {deville2011stability,
    AUTHOR = {DeVille, R. E. Lee and Sri Namachchivaya, N. and Rapti, Zoi},
     TITLE = {Stability of a stochastic two-dimensional non-{H}amiltonian
              system},
   JOURNAL = {SIAM J. Appl. Math.},
  FJOURNAL = {SIAM Journal on Applied Mathematics},
    VOLUME = {71},
      YEAR = {2011},
    NUMBER = {4},
     PAGES = {1458--1475},
      ISSN = {0036-1399,1095-712X},
   MRCLASS = {37H15 (34D08 60F17 60G35)},
  MRNUMBER = {2835242},
MRREVIEWER = {Peter\ E.\ Kloeden},
       DOI = {10.1137/100782139},
       URL = {https://doi.org/10.1137/100782139},
}

@article {chemnitz2023positive,
    AUTHOR = {Chemnitz, Dennis and Engel, Maximilian},
     TITLE = {Positive {L}yapunov exponent in the {H}opf normal form with
              additive noise},
   JOURNAL = {Comm. Math. Phys.},
  FJOURNAL = {Communications in Mathematical Physics},
    VOLUME = {402},
      YEAR = {2023},
    NUMBER = {2},
     PAGES = {1807--1843},
      ISSN = {0010-3616,1432-0916},
   MRCLASS = {37G05 (37D25)},
  MRNUMBER = {4627332},
       DOI = {10.1007/s00220-023-04764-z},
       URL = {https://doi.org/10.1007/s00220-023-04764-z},
}

@article {mohammed1997lyapunov,
    AUTHOR = {Mohammed, Salah-Eldin A. and Scheutzow, Michael K. R.},
     TITLE = {Lyapunov exponents of linear stochastic
              functional-differential equations. {II}. {E}xamples and case
              studies},
   JOURNAL = {Ann. Probab.},
  FJOURNAL = {The Annals of Probability},
    VOLUME = {25},
      YEAR = {1997},
    NUMBER = {3},
     PAGES = {1210--1240},
      ISSN = {0091-1798,2168-894X},
   MRCLASS = {60H10 (34K50 60H20 60H25)},
  MRNUMBER = {1457617},
       DOI = {10.1214/aop/1024404511},
       URL = {https://doi.org/10.1214/aop/1024404511},
}

@article {pinksy1988lyapunov,
    AUTHOR = {Pinsky, M. A. and Wihstutz, V.},
     TITLE = {Lyapunov exponents of nilpotent {I}t\^o{} systems},
   JOURNAL = {Stochastics},
  FJOURNAL = {Stochastics},
    VOLUME = {25},
      YEAR = {1988},
    NUMBER = {1},
     PAGES = {43--57},
      ISSN = {0090-9491},
   MRCLASS = {60H10},
  MRNUMBER = {1008234},
MRREVIEWER = {Ludwig\ Arnold},
       DOI = {10.1080/17442508808833531},
       URL = {https://doi.org/10.1080/17442508808833531},
}

@article {baxendale2001lyapunov,
    AUTHOR = {Baxendale, Peter H. and Goukasian, Levon},
     TITLE = {Lyapunov exponents of nilpotent {I}t\^o{} systems with random
              coefficients},
   JOURNAL = {Stochastic Process. Appl.},
  FJOURNAL = {Stochastic Processes and their Applications},
    VOLUME = {95},
      YEAR = {2001},
    NUMBER = {2},
     PAGES = {219--233},
      ISSN = {0304-4149,1879-209X},
   MRCLASS = {60H10 (37H15 60J60)},
  MRNUMBER = {1854026},
MRREVIEWER = {Olivier\ Raimond},
       DOI = {10.1016/S0304-4149(01)00091-6},
       URL = {https://doi.org/10.1016/S0304-4149(01)00091-6},
}

@article {goldsheid2022exponential,
    AUTHOR = {Goldsheid, Ilya},
     TITLE = {Exponential growth of products of non-stationary
              {M}arkov-dependent matrices},
   JOURNAL = {Int. Math. Res. Not. IMRN},
  FJOURNAL = {International Mathematics Research Notices. IMRN},
      YEAR = {2022},
    NUMBER = {8},
     PAGES = {6310--6346},
      ISSN = {1073-7928,1687-0247},
   MRCLASS = {60B15 (60J05)},
  MRNUMBER = {4406133},
MRREVIEWER = {G\"oran\ H\"ogn\"as},
       DOI = {10.1093/imrn/rnab269},
       URL = {https://doi.org/10.1093/imrn/rnab269},
}

@misc{gorodetski2024central,
      title={Central Limit Theorem for non-stationary random products of {$\SL(2, \R)$} matrices}, 
      author={Anton Gorodetski and Victor Kleptsyn and Grigorii Monakov},
      year={2024},
      eprint={2411.12003},
      archivePrefix={arXiv},
      primaryClass={math.PR},
      url={https://arxiv.org/abs/2411.12003}, 
}

@incollection {guivarch1984random,
    AUTHOR = {Guivarc'h, Yves and Raugi, Albert},
     TITLE = {Products of random matrices: convergence theorems},
 BOOKTITLE = {Random matrices and their applications ({B}runswick, {M}aine,
              1984)},
    SERIES = {Contemp. Math.},
    VOLUME = {50},
     PAGES = {31--54},
 PUBLISHER = {Amer. Math. Soc., Providence, RI},
      YEAR = {1986},
   MRCLASS = {60B15 (60F05 60F10)},
  MRNUMBER = {841080},
MRREVIEWER = {Philippe Bougerol},
       DOI = {10.1090/conm/050/841080},
       URL = {https://doi-org.proxy-um.researchport.umd.edu/10.1090/conm/050/841080},
}

@incollection {ledrappier1986positivity,
    AUTHOR = {Ledrappier, Fran{\c c}ois},
     TITLE = {Positivity of the exponent for stationary sequences of
              matrices},
 BOOKTITLE = {Lyapunov exponents ({B}remen, 1984)},
    SERIES = {Lecture Notes in Math.},
    VOLUME = {1186},
     PAGES = {56--73},
 PUBLISHER = {Springer, Berlin},
      YEAR = {1986},
      ISBN = {3-540-16458-8},
   MRCLASS = {60J15 (28D20 58F11 60B15 60F99)},
  MRNUMBER = {850070},
MRREVIEWER = {Michael\ Keane},
       DOI = {10.1007/BFb0076833},
       URL = {https://doi.org/10.1007/BFb0076833},
}

@article {baxendale1989lyapunov,
    AUTHOR = {Baxendale, Peter H.},
     TITLE = {Lyapunov exponents and relative entropy for a stochastic flow
              of diffeomorphisms},
   JOURNAL = {Probab. Theory Related Fields},
  FJOURNAL = {Probability Theory and Related Fields},
    VOLUME = {81},
      YEAR = {1989},
    NUMBER = {4},
     PAGES = {521--554},
      ISSN = {0178-8051,1432-2064},
   MRCLASS = {60H10 (58D99 58F11 60J60)},
  MRNUMBER = {995809},
MRREVIEWER = {Philippe\ Bougerol},
       DOI = {10.1007/BF00367301},
       URL = {https://doi.org/10.1007/BF00367301},
}

@article {lessa2021entropy,
    AUTHOR = {Lessa, Pablo},
     TITLE = {Entropy and dimension of disintegrations of stationary
              measures},
   JOURNAL = {Trans. Amer. Math. Soc. Ser. B},
  FJOURNAL = {Transactions of the American Mathematical Society. Series B},
    VOLUME = {8},
      YEAR = {2021},
     PAGES = {105--129},
      ISSN = {2330-0000},
   MRCLASS = {37C86 (37C45)},
  MRNUMBER = {4216247},
MRREVIEWER = {Thomas\ Ward},
       DOI = {10.1090/btran/60},
       URL = {https://doi.org/10.1090/btran/60},
}

@article {ledrappier2024exact,
    AUTHOR = {Ledrappier, Fran{\c c}ois and Lessa, Pablo},
     TITLE = {Exact dimension of dynamical stationary measures},
   JOURNAL = {J. Mod. Dyn.},
  FJOURNAL = {Journal of Modern Dynamics},
    VOLUME = {20},
      YEAR = {2024},
     PAGES = {679--715},
      ISSN = {1930-5311,1930-532X},
   MRCLASS = {37H15 (22D40 28A80 37A99 37C45)},
  MRNUMBER = {4845880},
       DOI = {10.3934/jmd.2024019},
       URL = {https://doi.org/10.3934/jmd.2024019},
}

@article {avila2010extremal,
    AUTHOR = {Avila, Artur and Viana, Marcelo},
     TITLE = {Extremal {L}yapunov exponents: an invariance principle and
              applications},
   JOURNAL = {Invent. Math.},
  FJOURNAL = {Inventiones Mathematicae},
    VOLUME = {181},
      YEAR = {2010},
    NUMBER = {1},
     PAGES = {115--189},
      ISSN = {0020-9910,1432-1297},
   MRCLASS = {37D30 (37A25 37C40 37D25)},
  MRNUMBER = {2651382},
MRREVIEWER = {Ian\ Melbourne},
       DOI = {10.1007/s00222-010-0243-1},
       URL = {https://doi.org/10.1007/s00222-010-0243-1},
}

@article {viana2008almost,
    AUTHOR = {Viana, Marcelo},
     TITLE = {Almost all cocycles over any hyperbolic system have
              nonvanishing {L}yapunov exponents},
   JOURNAL = {Ann. of Math. (2)},
  FJOURNAL = {Annals of Mathematics. Second Series},
    VOLUME = {167},
      YEAR = {2008},
    NUMBER = {2},
     PAGES = {643--680},
      ISSN = {0003-486X},
   MRCLASS = {37D25 (37A20 37C05 37C20)},
  MRNUMBER = {2415384},
MRREVIEWER = {Carlos H. V\'{a}squez},
       DOI = {10.4007/annals.2008.167.643},
       URL = {https://doi.org/10.4007/annals.2008.167.643},
}

@book {viana2014lectures,
    AUTHOR = {Viana, Marcelo},
     TITLE = {Lectures on {L}yapunov exponents},
    SERIES = {Cambridge Studies in Advanced Mathematics},
    VOLUME = {145},
 PUBLISHER = {Cambridge University Press, Cambridge},
      YEAR = {2014},
     PAGES = {xiv+202},
      ISBN = {978-1-107-08173-4},
   MRCLASS = {37-01 (34-01 34D08 37D25 37H05 37H15)},
  MRNUMBER = {3289050},
MRREVIEWER = {Paulo Varandas},
       DOI = {10.1017/CBO9781139976602},
       URL = {https://doi-org.proxy-um.researchport.umd.edu/10.1017/CBO9781139976602},
}

@article {avila2007simplicity,
    AUTHOR = {Avila, Artur and Viana, Marcelo},
     TITLE = {Simplicity of {L}yapunov spectra: a sufficient criterion},
   JOURNAL = {Port. Math. (N.S.)},
  FJOURNAL = {Portugaliae Mathematica. Nova S\'{e}rie},
    VOLUME = {64},
      YEAR = {2007},
    NUMBER = {3},
     PAGES = {311--376},
      ISSN = {0032-5155},
   MRCLASS = {37A20 (28D05 37D25 37F99)},
  MRNUMBER = {2350698},
MRREVIEWER = {Serge E. Troubetzkoy},
       DOI = {10.4171/PM/1789},
       URL = {https://doi-org.proxy-um.researchport.umd.edu/10.4171/PM/1789},
}

@article {benedicks2006random,
    AUTHOR = {Benedicks, Michael and Viana, Marcelo},
     TITLE = {Random perturbations and statistical properties of
              {H}\'enon-like maps},
   JOURNAL = {Ann. Inst. H. Poincar\'e{} C Anal. Non Lin\'eaire},
  FJOURNAL = {Annales de l'Institut Henri Poincar\'e{} C. Analyse Non
              Lin\'eaire},
    VOLUME = {23},
      YEAR = {2006},
    NUMBER = {5},
     PAGES = {713--752},
      ISSN = {0294-1449,1873-1430},
   MRCLASS = {37D25 (37A50 37C40 37H99)},
  MRNUMBER = {2259614},
MRREVIEWER = {Pei\ Dong\ Liu},
       DOI = {10.1016/j.anihpc.2004.10.013},
       URL = {https://doi.org/10.1016/j.anihpc.2004.10.013},
}

@article {young1986stochastic,
    AUTHOR = {Young, Lai-Sang},
     TITLE = {Stochastic stability of hyperbolic attractors},
   JOURNAL = {Ergodic Theory Dynam. Systems},
  FJOURNAL = {Ergodic Theory and Dynamical Systems},
    VOLUME = {6},
      YEAR = {1986},
    NUMBER = {2},
     PAGES = {311--319},
      ISSN = {0143-3857,1469-4417},
   MRCLASS = {58F15 (58F10 60J05)},
  MRNUMBER = {857204},
MRREVIEWER = {Massimo\ Campanino},
       DOI = {10.1017/S0143385700003473},
       URL = {https://doi.org/10.1017/S0143385700003473},
}

@incollection {peres1991analytic,
    AUTHOR = {Peres, Yuval},
     TITLE = {Analytic dependence of {L}yapunov exponents on transition
              probabilities},
 BOOKTITLE = {Lyapunov exponents ({O}berwolfach, 1990)},
    SERIES = {Lecture Notes in Math.},
    VOLUME = {1486},
     PAGES = {64--80},
 PUBLISHER = {Springer, Berlin},
      YEAR = {1991},
      ISBN = {3-540-54662-6},
   MRCLASS = {60J15 (60B15)},
  MRNUMBER = {1178947},
MRREVIEWER = {Albert\ Raugi},
       DOI = {10.1007/BFb0086658},
       URL = {https://doi.org/10.1007/BFb0086658},
}

@article {ruelle1979ergodic,
    AUTHOR = {Ruelle, David},
     TITLE = {Ergodic theory of differentiable dynamical systems},
   JOURNAL = {Inst. Hautes \'Etudes Sci. Publ. Math.},
  FJOURNAL = {Institut des Hautes \'Etudes Scientifiques. Publications
              Math\'ematiques},
    NUMBER = {50},
      YEAR = {1979},
     PAGES = {27--58},
      ISSN = {0073-8301,1618-1913},
   MRCLASS = {58F18 (58F15 60G99 82A05)},
  MRNUMBER = {556581},
MRREVIEWER = {L.\ A.\ Bunimovich},
       URL = {http://www.numdam.org/item?id=PMIHES_1979__50__27_0},
}

@article {blumenthal2022positive,
    AUTHOR = {Blumenthal, Alex and Yang, Yun},
     TITLE = {Positive {L}yapunov exponent for random perturbations of
              predominantly expanding multimodal circle maps},
   JOURNAL = {Ann. Inst. H. Poincar\'e{} C Anal. Non Lin\'eaire},
  FJOURNAL = {Annales de l'Institut Henri Poincar\'e{} C. Analyse Non
              Lin\'eaire},
    VOLUME = {39},
      YEAR = {2022},
    NUMBER = {2},
     PAGES = {419--455},
      ISSN = {0294-1449,1873-1430},
   MRCLASS = {37E10 (37A50 37C40 37D25 37H15)},
  MRNUMBER = {4412073},
MRREVIEWER = {Simon\ Lloyd},
       DOI = {10.4171/aihpc/11},
       URL = {https://doi.org/10.4171/aihpc/11},
}

@article {young2008chaotic,
    AUTHOR = {Young, Lai-Sang},
     TITLE = {Chaotic phenomena in three settings: large, noisy and out of
              equilibrium},
   JOURNAL = {Nonlinearity},
  FJOURNAL = {Nonlinearity},
    VOLUME = {21},
      YEAR = {2008},
    NUMBER = {11},
     PAGES = {T245--T252},
      ISSN = {0951-7715,1361-6544},
   MRCLASS = {37-02 (37D45 37H10 82C99)},
  MRNUMBER = {2448225},
       DOI = {10.1088/0951-7715/21/11/T04},
       URL = {https://doi.org/10.1088/0951-7715/21/11/T04},
}

@article {young1986random,
    AUTHOR = {Young, L.-S.},
     TITLE = {Random perturbations of matrix cocycles},
   JOURNAL = {Ergodic Theory Dynam. Systems},
  FJOURNAL = {Ergodic Theory and Dynamical Systems},
    VOLUME = {6},
      YEAR = {1986},
    NUMBER = {4},
     PAGES = {627--637},
      ISSN = {0143-3857,1469-4417},
   MRCLASS = {28D05 (15A51 54H20 58F11)},
  MRNUMBER = {873436},
MRREVIEWER = {Eric\ V.\ Slud},
       DOI = {10.1017/S0143385700003734},
       URL = {https://doi.org/10.1017/S0143385700003734},
}

@book {kifer1988random,
    AUTHOR = {Kifer, Yuri},
     TITLE = {Random perturbations of dynamical systems},
    SERIES = {Progress in Probability and Statistics},
    VOLUME = {16},
 PUBLISHER = {Birkh\"auser Boston, Inc., Boston, MA},
      YEAR = {1988},
     PAGES = {vi+294},
      ISBN = {0-8176-3384-7},
   MRCLASS = {58F30 (60F10 60J99)},
  MRNUMBER = {1015933},
MRREVIEWER = {Vadim\ A.\ Ka\u imanovich},
       DOI = {10.1007/978-1-4615-8181-9},
       URL = {https://doi.org/10.1007/978-1-4615-8181-9},
}

@article {furstenberg1963noncommuting,
    AUTHOR = {Furstenberg, Harry},
     TITLE = {Noncommuting random products},
   JOURNAL = {Trans. Amer. Math. Soc.},
  FJOURNAL = {Transactions of the American Mathematical Society},
    VOLUME = {108},
      YEAR = {1963},
     PAGES = {377--428},
      ISSN = {0002-9947},
   MRCLASS = {60.08 (60.66)},
  MRNUMBER = {163345},
MRREVIEWER = {G.-C. Rota},
       DOI = {10.2307/1993589},
       URL = {https://doi-org.proxy-um.researchport.umd.edu/10.2307/1993589},
}

@article {pollicott2010maximal,
    AUTHOR = {Pollicott, Mark},
     TITLE = {Maximal {L}yapunov exponents for random matrix products},
   JOURNAL = {Invent. Math.},
  FJOURNAL = {Inventiones Mathematicae},
    VOLUME = {181},
      YEAR = {2010},
    NUMBER = {1},
     PAGES = {209--226},
      ISSN = {0020-9910,1432-1297},
   MRCLASS = {37H15 (60B20 60F15)},
  MRNUMBER = {2651384},
MRREVIEWER = {Mark\ W.\ Meckes},
       DOI = {10.1007/s00222-010-0246-y},
       URL = {https://doi.org/10.1007/s00222-010-0246-y},
}

@article {blumenthal2018lyapunov,
    AUTHOR = {Blumenthal, Alex and Xue, Jinxin and Young, Lai-Sang},
     TITLE = {Lyapunov exponents and correlation decay for random
              perturbations of some prototypical 2{D} maps},
   JOURNAL = {Comm. Math. Phys.},
  FJOURNAL = {Communications in Mathematical Physics},
    VOLUME = {359},
      YEAR = {2018},
    NUMBER = {1},
     PAGES = {347--373},
      ISSN = {0010-3616,1432-0916},
   MRCLASS = {37H10 (37D25)},
  MRNUMBER = {3781453},
MRREVIEWER = {M.\ L.\ Blank},
       DOI = {10.1007/s00220-017-2999-2},
       URL = {https://doi.org/10.1007/s00220-017-2999-2},
}

@article {bednarski2025lyapunov,
    AUTHOR = {Bednarski, Sam and Quas, Anthony},
     TITLE = {Lyapunov exponents of orthogonal-plus-normal cocycles},
   JOURNAL = {Nonlinearity},
  FJOURNAL = {Nonlinearity},
    VOLUME = {38},
      YEAR = {2025},
    NUMBER = {3},
     PAGES = {Paper No. 035007, 17},
      ISSN = {0951-7715,1361-6544},
   MRCLASS = {37H15},
  MRNUMBER = {4863648},
       DOI = {10.1088/1361-6544/adac9a},
       URL = {https://doi.org/10.1088/1361-6544/adac9a},
}

@book {knapp2002lie,
    AUTHOR = {Knapp, Anthony W.},
     TITLE = {Lie groups beyond an introduction},
    SERIES = {Progress in Mathematics},
    VOLUME = {140},
   EDITION = {Second},
 PUBLISHER = {Birkh\"auser Boston, Inc., Boston, MA},
      YEAR = {2002},
     PAGES = {xviii+812},
      ISBN = {0-8176-4259-5},
   MRCLASS = {22-01},
  MRNUMBER = {1920389},
}

@article {abels1995semigroups,
    AUTHOR = {Abels, H. and Margulis, G. A. and So{\u i}fer, G. A.},
     TITLE = {Semigroups containing proximal linear maps},
   JOURNAL = {Israel J. Math.},
  FJOURNAL = {Israel Journal of Mathematics},
    VOLUME = {91},
      YEAR = {1995},
    NUMBER = {1-3},
     PAGES = {1--30},
      ISSN = {0021-2172,1565-8511},
   MRCLASS = {22E45 (20G20 22D40 60B15)},
  MRNUMBER = {1348303},
MRREVIEWER = {B.\ Sury},
       DOI = {10.1007/BF02761637},
       URL = {https://doi.org/10.1007/BF02761637},
}

@article {armentano2024random,
    AUTHOR = {Armentano, Diego and Chinta, Gautam and Sahi, Siddhartha and
              Shub, Michael},
     TITLE = {Random and mean {L}yapunov exponents for {${\rm GL}_n(\mathbb{R})$}},
   JOURNAL = {Ergodic Theory Dynam. Systems},
  FJOURNAL = {Ergodic Theory and Dynamical Systems},
    VOLUME = {44},
      YEAR = {2024},
    NUMBER = {8},
     PAGES = {2063--2079},
      ISSN = {0143-3857,1469-4417},
   MRCLASS = {37H15 (22E45 37D25)},
  MRNUMBER = {4798790},
       DOI = {10.1017/etds.2023.106},
       URL = {https://doi.org/10.1017/etds.2023.106},
}

@article {rivin2005mean,
    AUTHOR = {Rivin, Igor},
     TITLE = {On some mean matrix inequalities of dynamical interest},
   JOURNAL = {Comm. Math. Phys.},
  FJOURNAL = {Communications in Mathematical Physics},
    VOLUME = {254},
      YEAR = {2005},
    NUMBER = {3},
     PAGES = {651--658},
      ISSN = {0010-3616,1432-0916},
   MRCLASS = {37A25 (15A45 37C40)},
  MRNUMBER = {2126486},
       DOI = {10.1007/s00220-004-1282-5},
       URL = {https://doi.org/10.1007/s00220-004-1282-5},
}

@incollection {burns1999recent,
    AUTHOR = {Burns, Keith and Pugh, Charles and Shub, Michael and
              Wilkinson, Amie},
     TITLE = {Recent results about stable ergodicity},
 BOOKTITLE = {Smooth ergodic theory and its applications ({S}eattle, {WA},
              1999)},
    SERIES = {Proc. Sympos. Pure Math.},
    VOLUME = {69},
     PAGES = {327--366},
 PUBLISHER = {Amer. Math. Soc., Providence, RI},
      YEAR = {2001},
      ISBN = {0-8218-2682-4},
   MRCLASS = {37D30 (37C20 37C40)},
  MRNUMBER = {1858538},
MRREVIEWER = {Pei\ Dong\ Liu},
       DOI = {10.1090/pspum/069/1858538},
       URL = {https://doi.org/10.1090/pspum/069/1858538},
}

@book {viana2016foundations,
    AUTHOR = {Viana, Marcelo and Oliveira, Krerley},
     TITLE = {Foundations of ergodic theory},
    SERIES = {Cambridge Studies in Advanced Mathematics},
    VOLUME = {151},
 PUBLISHER = {Cambridge University Press, Cambridge},
      YEAR = {2016},
     PAGES = {xvi+530},
      ISBN = {978-1-107-12696-1},
   MRCLASS = {37-02 (28D05 37Axx 37D35 54H20)},
  MRNUMBER = {3558990},
MRREVIEWER = {Douglas\ P.\ Dokken},
       DOI = {10.1017/CBO9781316422601},
       URL = {https://doi.org/10.1017/CBO9781316422601},
}

@incollection{dedieu2003shub,
    AUTHOR = {Dedieu, Jean-Pierre and Shub, Mike},
     TITLE = {On random and mean exponents for unitarily invariant
              probability measures on {$\mathbb{GL}_n(\mathbb C)$}},
      NOTE = {Geometric methods in dynamics. II},
   JOURNAL = {Ast\'erisque},
  FJOURNAL = {Ast\'erisque},
    NUMBER = {287},
      YEAR = {2003},
     PAGES = {xvii, 1--18},
      ISSN = {0303-1179,2492-5926},
   MRCLASS = {37H15 (37A50 43A80)},
  MRNUMBER = {2039997},
MRREVIEWER = {Yves\ Lacroix},
}

@misc{gorodetski2023nonstationary,
      title={Non-stationary version of {F}urstenberg Theorem on random matrix products}, 
      author={Anton Gorodetski and Victor Kleptsyn},
      year={2023},
      eprint={2210.03805},
      archivePrefix={arXiv},
      primaryClass={math.DS},
      url={https://arxiv.org/abs/2210.03805}, 
}

\end{document}